\documentclass{amsart}

\usepackage{amsmath}
\usepackage{amsthm}
\usepackage{amssymb}
\usepackage{amscd}
\usepackage{bm}
\usepackage{cite}
\usepackage{stmaryrd}
\usepackage{amsfonts} 
\usepackage{MnSymbol}
\DeclareMathSymbol{\sm}{\mathbin}{AMSa}{"39}

\usepackage{enumitem}

\DeclareFontFamily{U}{wncy}{}
\DeclareFontShape{U}{wncy}{m}{n}{<->wncyr6}{}
\DeclareSymbolFont{mcy}{U}{wncy}{m}{n}
\DeclareMathSymbol{\Sh}{\mathord}{mcy}{"58} 

\usepackage{tikz}
\usetikzlibrary{cd}
\usepackage{mathtools}

\newcommand{\ldef}{{\rm MC}(\mathcal{L})}
\newcommand{\ldefp}{{\rm MC}(\mathcal{L}^\prime)}

\newtheorem{thm}{Theorem}
\numberwithin{thm}{section}
\newtheorem{lem}{Lemma}
\numberwithin{lem}{section}
\newtheorem{prop}{Proposition}
\numberwithin{prop}{section}
\newtheorem{cor}{Corollary}
\numberwithin{cor}{section}

\newtheorem{manualtheoreminner}{Corollary}
\newenvironment{manualtheorem}[1]{%
  \renewcommand\themanualtheoreminner{#1}%
  \manualtheoreminner
}{\endmanualtheoreminner}

\newtheorem{manualpropinner}{Proposition}
\newenvironment{manualprop}[1]{%
  \renewcommand\themanualpropinner{#1}%
  \manualpropinner
}{\endmanualpropinner}

\theoremstyle{definition}
\newtheorem{defn}{Definition}
\numberwithin{defn}{section}
\newtheorem{exam}{Example}
\numberwithin{exam}{section}
\theoremstyle{remark}
\newtheorem{rem}{Remark}
\numberwithin{rem}{section}

\title{Poisson structures on sets of Maurer-Cartan elements}
\author{Thomas Machon}
\address{H.H.~Wills Physics Laboratory, Tyndall Avenue, Bristol BS8 1TL, UK}
\email{t.machon@bristol.ac.uk}
\begin{document}

\maketitle

\begin{abstract}
Given a differential graded Lie algebra (dgla) $\mathcal{L}$ satisfying certain conditions, we construct Poisson structures on the gauge orbits of its set of Maurer-Cartan (MC) elements, termed Maurer-Cartan-Poisson (MCP) structures. They associate a compatible Batalin-Vilkovisky algebra to each MC element of $\mathcal{L}$. An MCP structure is shown to exist for a number of dglas associated to commutative Frobenius algebras, deformations of Poisson and symplectic structures, as well as the Chevally-Eilenberg complex. MCP structures yield a notion of hamiltonian flow of MC elements, and also define Lie algebroids on gauge orbits, whose isotropy algebras give invariants of MC elements. As an example, this gives a finite-dimensional two-step nilpotent graded Lie algebra associated to any closed symplectic manifold.
\end{abstract}

\tableofcontents

\newpage
\section{Introduction}

Over a field $k$ of characteristic zero it is generally true~\cite{pridham2010unifying} that a problem in deformation theory is governed by a differential graded Lie algebra (dgla) $\mathcal{L} = (L, d, [\cdot, \cdot])$, with underlying graded vector space $L = \bigoplus_{i \in \mathbb{Z}} L^i$. In the dgla formalism, deformations of an object of interest are given by the set of Maurer-Cartan elements of $\mathcal{L}$, defined as
$${\rm MC}(\mathcal{L}) = \left \{ X \in L^1 \; | \; dX+\frac{1}{2}[X,X]=0 \right \} .$$
The vector space $L^1$ has an action of the gauge Lie algebra $L^0$. This action preserves the set $\ldef$ and so decomposes $\ldef$ into gauge orbits. In deformation theory (see e.g. Refs.~\cite{manetti2004deformation,goldman1988deformation}) one is typically interested in the space of orbits, elements of $\ldef$ up to gauge equivalence. Here we focus on the orbits themselves, our goal is to show that for a number of differential graded Lie algebras, the gauge orbits of $\ldef$ each possess a Poisson structure, and their collection defines (roughly) a Poisson structure on the set $\ldef$. We call these Maurer-Cartan-Poisson (MCP) structures. They are strongly reminiscent of the Kirillov-Kostant-Souriau (KKS) symplectic forms on the coadjoint orbits of the dual to a Lie algebra~\cite{kirillov2004lectures,marsden2013introduction}, and reduce to them in certain cases.
 
In the examples we construct, a dgla $\mathcal{L}= (L, d, [\cdot, \cdot])$ must satisfy two requirements for an MCP structure to exist. The first is that $\ldef$ is a cone (if $X \in \ldef$, then $\lambda X \in \ldef$ for all $\lambda \in k$), the second is that the shifted dual space $(L^\vee)[1]$ (interpreted appropriately in the infinite-dimensional setting) is a graded commutative algebra, such that for each $X \in \ldef$, the adjoint of $d_X = d+[X, \cdot]$ defines a Batalin-Vilkovisky (BV) algebra on $(L^\vee)[1]$, compatible with $\mathcal{L}$ in a certain way. The requirement that $\ldef$ be a cone restricts the dglas for which an MCP structure may be defined, in all our examples they either have trivial Lie bracket (in which case the MCP structure reduces to a Lie-Poisson structure) or the dgla has trivial differential. The latter case can be achieved by taking the cohomology of a dgla, which is a graded Lie algebra, the approach taken in Section~\ref{sec:frobenius}, where we construct an MCP structure for the set of Poisson algebras compatible with a given commutative Frobenius algebra. Other examples of MCP structures, given in Sections~\ref{sec:examlie}-\ref{sec:examsym}, are the Chevally-Eilenberg complex and two infinite-dimensional cases, deformations of Poisson and symplectic structures.

We now give some example results. Throughout everything will be over a field $k$ which will be either $\mathbb{R}$ or $\mathbb{C}$. Let $A$ be a finite-dimensional commutative Frobenius algebra. Let $\mathcal{P}(A)$ be the set of all Lie brackets on $A$ which make $A$ a Poisson algebra. We show in Section~\ref{sec:frobenius} that this is a cone in the second Hochschild cohomology group of $A$, ${\rm HH}^2(A,A)$, defined by the vanishing of a set of homogeneous quadratic polynomials. We construct a dgla for this problem, and an associated MCP structure. The following are then corollaries of Theorem~\ref{thm:frobenius}.
\begin{cor}
The coordinate ring of $\mathcal{P}(A)$ has the structure of a Poisson algebra.
\end{cor}
The dgla gives a gauge action on $\mathcal{P}(A)$ decomposing it into gauge orbits $\mathcal{O}$, with each gauge orbit $\mathcal{O} \subset \mathcal{P}(A)$ a smooth immersed submanifold of ${\rm HH}^2(A,A)$. 
\begin{cor}
Each gauge orbit $\mathcal{O} \subset \mathcal{P}(A) \subset {\rm HH}^2(A,A)$ is a Poisson manifold.
\end{cor}
 \newpage
If the base dgla of the MCP structure has trivial Lie bracket, the set $\ldef$ is a manifold, and we obtain a Lie-Poisson structure from our construction. A simple example is the Chevally-Eilenberg complex, for which we construct an MCP structure in Section~\ref{sec:examlie}.
\begin{prop}
\label{prop:lie}
Let $\mathfrak{g}$ be a finite-dimensional Lie algebra. The space of central extensions of $\mathfrak{g}$ has a Lie-Poisson structure given by the Lie algebra defined by the short exact sequence
\begin{center}
\begin{tikzcd}
\mathfrak{a}^{b_2} \arrow[hookrightarrow, r] & \bigwedge \nolimits ^2 \mathfrak{g} \big / \delta_\mathfrak{g} \bigwedge \nolimits ^3 \mathfrak{g} \arrow[r, twoheadrightarrow ] & \left [ \mathfrak{g}, \mathfrak{g} \right ],
 \end{tikzcd}
\end{center}
with the Poisson structure on each gauge orbit being the Lie-Poisson structure associated to $[\mathfrak{g}, \mathfrak{g}]$.
\end{prop}

The remaining properties of MCP structures can be illustrated with the example of volume-preserving deformations of symplectic structures, explored in Section~\ref{sec:examsym}, which continue from the work in Ref.~\cite{machon2020poisson}. Let $(M, \omega)$ be a closed symplectic manifold. We first recall the definition of the symplectic differential (see e.g.\cite{tseng2012cohomology}) acting on differential forms
$$ d^\Lambda = [\iota_\pi, d],$$
where $\iota_\pi$ is contraction with the Poisson structure $\pi$ defined as the inverse of $\omega$ and $[\cdot, \cdot]$ is the graded commutator. An MCP structure induces a notion of hamiltonian flow of Maurer-Cartan elements, we show that for the MCP structure associated to deformations of symplectic structures this flow equation is given as
$$\partial_t \omega = d d^\Lambda \beta,$$
where $\beta$ is a 2-form arising as the derivative of some hamiltonian functional on the space of symplectic structures. This equation is strongly reminiscent of the geometric flow for the type IIA string defined by Fei et al.~\cite{fei2020geometric}, building on the work of Hitchin~\cite{hitchin2000geometry}. 

MCP structures define Poisson structures on gauge orbits. Given a particular gauge orbit $\mathcal{O}$, the Poisson structure on $\mathcal{O}$ may or may not be symplectic (non-degenerate) and this may change between different orbits which are nevertheless in the same set $\ldef$. An example is given by the MCP structure associated to deformations of a closed symplectic manifold that preserve the symplectic volume form. Let $\mathcal{O}_\omega$ be the gauge orbit containing $\omega$, then we have the following Corollary of Theorem~\ref{thm:sym}. 
\begin{cor}
The Poisson structure on $\mathcal{O}_\omega$ is symplectic if and only if 
$$ {\rm dim}\; H_{dd^\Lambda}^2(M) = b^2(M).$$
\end{cor}
\begin{rem}
On a 4-manifold is this equivalent to the hard Lefschetz condition.
\end{rem}
$ H_{dd^\Lambda}^2(M) $ is one of the finite-dimensional cohomology groups for closed symplectic manifolds defined by Tseng and Yau~\cite{tseng2012cohomology}, built out of differential forms using the differentials $d$ and $d^\Lambda$. These cohomology groups satisfy the inequalities
$$ {\rm dim} \; H_{dd^\Lambda}^i(M) \geq b^i(M),$$
where $b^i(M)$ are the Betti numbers. Equality holds for all $i$ if and only if $(M, \omega)$ satisfies the hard Lefschetz condition (see e.g. Ref.~\cite{angella2015inequalities}).

A Poisson structure defines a Lie algebroid, correspondingly MCP structures define Lie algebroids on gauge orbits. A notable property of a Lie algebroid is its collection of isotropy Lie algebras, these are Lie algebras associated to each point of the base manifold. In our case we obtain isotropy algebras for MC elements. In the symplectic case we have the following corollary of Theorem~\ref{thm:sym}.
\begin{cor}
Associated to any closed symplectic $2n$-manifold $(M, \omega)$ is a finite-dimensional two-step nilpotent graded Lie algebra 
$$\mathfrak{k} = \bigoplus_{i=0}^{2n-4} \mathfrak{k}^i,$$
with $\mathfrak{k}^0$ the isotropy algebra at $\omega$ of the Lie algebroid on the gauge orbit containing $\omega$. $\mathfrak{k}$ is zero-dimensional if and only if $(M, \omega)$ satisfies the hard Lefschetz condition.
\end{cor}

This paper follows from a previous work~\cite{machon2020poisson}, in which it was observed that there is a natural Poisson bracket on the set of Poisson structures on a given closed manifold. 
\thanks{I would like to thank Daniel~Plummer for helpful conversations, Nicoletta~Tardini for useful discussions about the $dd^\Lambda$ Lemma, and the anonymous referees for Ref.~\cite{machon2020poisson}, whose suggestion of the potential relevance to deformation quantization motivated this work.}

\section{Preliminaries} \label{sec:prelim}

In this section we establish notation and recall basic definitions of many of the objects we use. Throughout we will work over a field $k$ taken to be either $\mathbb{R}$ or $\mathbb{C}$, and for a graded vector space $V = \bigoplus_{i \in \mathbb{Z}} V^i$ we denote by $V[s]$ the shifted space, $V[s]^i = V^{i+s}$.  We are primarily concerned with the interaction of two algebraic structures on graded vector spaces, differential graded Lie algebras (dlgas) and Batalin-Vilkovisky (BV) algebras, we recall their definition here, as well as that of Lie algebroids.
\subsection{Differential graded Lie algebras} \label{sec:dgla}

A differential graded Lie algebra (dgla) $\mathcal{L}= (L, d, [\cdot, \cdot ])$ is a graded vector space
$$ L = \bigoplus_{i \in \mathbb{Z}} L^i,$$
equipped with a differential $d : L^i \to L^{i+1}$ and a bracket $[\cdot, \cdot]: L^i \times L^j \to L^{i+j}$ satisfying: 
\begin{itemize}[label={}]
\item graded skew-commutativity,  $[a,b] = (-1)^{|a||b|+1}[b,a]$,
\item Leibniz rule, $d[a, b] = [d a, b] + (-1)^{|a|} [a, db]$,
\item Jacobi identity, $(-1)^{|a||c|}[a,[b,c]]+(-1)^{|b||a|}[b,[c,a]]+(-1)^{|c||b|}[c,[a,b]] = 0$,
\end{itemize}
where for a homogeneous element $a \in L^i$, $\left | a \right | = i$. The space $L^0$ is a subalgebra of the dgla with respect to the bracket $[\cdot, \cdot]$, and is the gauge Lie algebra of the dgla. Associated to a dgla is its set of Maurer-Cartan  elements ${\rm MC}(\mathcal{L})$ defined as
$$ {\rm MC}(\mathcal{L}) = \left \{ a \in L^1\, | da + \frac{1}{2} [a,a] = 0 \right \}.$$
The gauge Lie algebra acts on the space $L^1$, and there is a Lie algebra homomorphism
$$ v : L^0 \to \Gamma(TL^1),$$
from $L^0$ into affine vector fields on $L^1$ given by
$$ v: \lambda \mapsto  d\lambda +  [X, \lambda ].$$
The flow of these vector fields preserves the set $\ldef$, and decomposes it into gauge orbits (see e.g.~\cite{goldman1988deformation}). Correspondingly, one can define a gauge transformation between two elements $X$ and $Y$ as solutions of the equation (see e.g.~\cite{hinich1996descent}, Proposition 2.2.3)
$$ \frac{dZ}{dt} = d \lambda+ [Z, \lambda] , \; Z(0)=X, \; Z(1)=Y.$$
If the dgla is nilpotent, this becomes an action of the gauge group ${\rm exp}(L^0)$ on $\ldef$, and leads to the definition of the Deligne groupoid (see e.g.~\cite{goldman1988deformation,getzler2002darboux}).
\subsection{Batalin-Vilkovisky algebras}

A Batalin-Vilkovisky (BV) algebra $\mathcal{B} = (B, \delta, \wedge)$ is a graded vector space
$$ B = \bigoplus_{i \in \mathbb{Z}} B^i,$$
equipped with a differential $\delta: B^i \to B^{i-1}$ and an associative product $\wedge : B^i \times B^j \to B^{i+j}$ such that:
\begin{itemize}[label = {}]
\item the product is graded commutative, $a\wedge b = (-1)^{|a| |b|} b \wedge a$;
\item the bracket on $B$,
$$ [a, b]_\mathcal{B} = (-1)^{|a|}\big (\delta(a \wedge b) -  \delta a \wedge b - (-1)^{|a|} a \wedge \delta b \big ),$$
measuring the failure of $\delta$ to be a derivation defines a graded Lie bracket satisfying the Jacobi identity
$$(-1)^{(|c|-1)(|b|-1)}[c,[a,b]_\mathcal{B}]_\mathcal{B}+ (-1)^{(|b|-1)(|a|-1)} [b,[c,a]_\mathcal{B}]_\mathcal{B}+ (-1)^{(|a|-1)(|c|-1)} [a,[b,c]_\mathcal{B}]_\mathcal{B}=0.$$
\end{itemize}
In particular this implies that the product and differential satisfy the order-2 nilpotency condition
\begin{multline*} \delta(a\wedge b \wedge c) - \delta(a \wedge b) \wedge c + \delta( a) \wedge b \wedge c - (-1)^{|a|} a \wedge \delta (b \wedge c) - \\
(-1)^{(|a|+1)|b|} b \wedge \delta( a \wedge c) + (-1)^{|a|} a \wedge (\delta b) \wedge c + (-1)^{|a|+|b|} a \wedge b \wedge \delta c = 0.
\end{multline*}
\subsection{Lie algebroids}

A Lie algebroid on a smooth manifold $M$ is a vector bundle $A$ over $M$ equipped with a Lie bracket on its space of sections and a morphism $a: A \to TM$ (the anchor) satisfying the Leibniz rule
$$[X, f Y] = Y \mathcal{L}_{a(X)}f + f[X,Y],$$
where $\mathcal{L}$ denotes the Lie derivative and $f \in C^\infty(M)$. Notably this implies~\cite{kosmann1990poisson} that $a$ is a homomorphism of Lie brackets.  A Poisson Lie algebroid is the Lie algebroid associated to a Poisson manifold with Poisson tensor $\Pi$. The vector bundle $A$ is the cotangent bundle $T^\ast M$, and the anchor $\Pi^\# : T^\ast M \to TM$ is the map $\alpha \mapsto \Pi(\alpha, \cdot)$ induced by the Poisson tensor. The bracket is given by
$$[\alpha, \beta] = \mathcal{L}_{\Pi^\#(\alpha)} \beta -  \mathcal{L}_{\Pi^\#(\beta)} \alpha - d \Pi(\alpha, \beta),$$
where $d$ is the de Rham differential.
\section{MCP structures} \label{sec:KS}

\begin{rem}
The definition of MCP structures was driven by the computation of the examples discussed in this paper, and should be regarded as a working definition.
\end{rem}
Recall we work over a field $k$ taken to be either $\mathbb{R}$ or $\mathbb{C}$, though the definition below makes sense for any field of characteristic zero.

\begin{defn} \label{def:mcp}
A {\bf Maurer-Cartan-Poisson (MCP) structure} with base dgla $\mathcal{L}$ is a triple, $(\mathcal{L}, \mathcal{B}, (\cdot, \cdot))$, consisting of a dgla $\mathcal{L}=(L, d, [\cdot, \cdot ]) $, a graded commutative algebra $\mathcal{B} = (B, \wedge)$, and a set of non-degenerate bilinear pairings $(\cdot, \cdot) :  L^i\times B^{i+1} \to k$, such that for each $X \in \ldef$: 
\begin{enumerate}[label={\arabic*.}]
\item $d_X X =0$ (equivalently $\ldef$ is a cone);
\item the differential $\delta_X: B^{i} \to B^{i-1}$ defined as
$$(d_X\; \cdot, \cdot) = (\cdot, \delta_X \; \cdot)$$
gives the triple $\mathcal{B}_X = (B, \delta_X, \wedge)$ the structure of a BV algebra;
\item there is an element $k_X \in k$ such that the homomorphism $\rho_X: B^1 \to L^0$ defined by
$$(X, a \wedge b) =(\rho_X a, b),$$
satisfies
$$ [\rho_X a, \rho_X b] = k_X \rho_X [a,b]_{\mathcal{B}_X},$$
where $[\cdot, \cdot]$ is the Lie bracket on $L^0$ and $[\cdot ,\cdot]_{\mathcal{B}_X}$ is the Gerstenhaber bracket on $\mathcal{B}_X$.
\end{enumerate}
\end{defn}
\begin{rem}
Definition~\ref{def:mcp} can be modified such that (1)-(3) are only valid for $X $ in a subset of gauge orbits of the base dlga. 
\end{rem}
The definition as given requires $\ldef$ to be a cone, which tightly constrains the possible dglas that can act as the base of an MCP structure.
\begin{lem}\label{lem:root}
Maurer-Cartan elements of the base dgla $\mathcal{L}$ of an MCP structure satisfy
$$dX = [X,X]=0.$$
\end{lem}
\begin{proof}
We require $d_X X=0$, or $d X+ [X,X]=0$. Since $X$ is a Maurer-Cartan element $dX+[X,X]/2=0$, this implies $[X,X]=dX=0$. 
\end{proof}
\begin{rem}
The element $X$ in property 1 should be thought of as a vector field on $L^1$, the image of the diagonal map $L^1 \to L^1 \times L^1$ with the identification $T L^1 \cong L^1 \times L^1$. In principle one could choose some other $d_X$ closed vector field on $L^1$, whether this can be done so as to define a Poisson structure on the gauge orbits of $\ldef$ is not clear. 
\end{rem}
\begin{rem} \label{rem:types}
The MCP structures we have found come in two flavors. The first have a base dgla with trivial Lie bracket, and the element $k_X=0$. This is the case for the Chevally-Eilenberg complex and deformations of symplectic structures (and this construction can be given for any BV algebra). The second have a base dgla with trivial differential, $d=0$, and $k_X\neq0$. This is the case for deformations of Poisson structures and Frobenius algebras.
\end{rem}
\begin{rem}
If the base dgla has trivial differential, then $\mathcal{L}$ is a graded Lie algebra, and (in finite-dimensions) the MCP structure is determined by a compatible commutative algebra on $(L^\vee)[1]$. It seems likely that this is related to Koszul duality.
\end{rem}

In our discussion of constant volume deformations of symplectic structures (Section~\ref{sec:cv}) it is necessary to consider a differential graded Lie subalgebra (sub-dgla) which is compatible with a given MCP structure. Recall that a sub-dgla ${\mathcal{L}}^\prime$ of a dgla $\mathcal{L}$ is subspace $L^\prime \subset L$ closed under both the differential $d$ and bracket $[\cdot, \cdot]$ of $\mathcal{L}$.

\begin{defn} \label{def:comp}
Let $(\mathcal{L}, \mathcal{B}, (\cdot, \cdot))$ be a MCP structure. A sub-dgla $\mathcal{L}^\prime$ is compatible with $(\mathcal{L}, \mathcal{B}, (\cdot, \cdot))$ if, for each $X \in {\rm MC}(\mathcal{L}^\prime)$, the image of the map $\nu_X = d_X \rho_X \delta_X : B^2 \to L^1$
is in $L^\prime$.
\end{defn}
\begin{rem}
We shall see that this approximately the statement that ${\rm MC}(\mathcal{L}^\prime)$ is a `Poisson submanifold' of $\ldef$.
\end{rem}
Given sub-dgla $\mathcal{L}^\prime$, for each $X \in {\rm MC}(\mathcal{L}^\prime)$  define the orthogonal complement $N_X \subset B$ of $(L^\prime)$ given as,
$$N_X = \bigoplus_i N_X^i,$$
with
$$N^i_X = \left \{ \alpha \in B \; | \; (U, \alpha)=0 ,\; {\rm for} \;{\rm  all} \; X \in (L_X^\prime)^i \right \}$$
\begin{lem}
If $\mathcal{L}^\prime$ is a compatible sub-dgla, then there is an exact sequence
\begin{center}
\begin{tikzcd}
N_X^2 \arrow[r, hookrightarrow, "\iota"] &B^2 \arrow[r, "\nu_X" ] & L^1 \arrow[r, twoheadrightarrow,"\pi" ] & L^1/(L^\prime)^1,
 \end{tikzcd}
 \end{center}
where $\iota$ and $\pi$ are the natural inclusion and projection.
\end{lem}
\begin{proof}
Given two elements of $B^2$, $a$ and $b$ and $X \in {\rm MC}(\mathcal{L}^\prime)$, there is an anti-symmetric pairing
$$ (X, \delta_X a \wedge \delta_X b) =(\nu_X a, b)= (-\nu_X b, a)$$
Then for $a \in B^1$ and $b \in N_X^2$, this vanishes, implying $\nu_Xb =0$.
\end{proof}

\section{Poisson algebras for MCP structures} \label{sec:POI}

A central observation is that MCP structures turn functions on $\ldef$ (defined as a quotient) into a Poisson algebra.

\subsection{Finite-dimensional case}

We first consider a finite-dimensional MCP structure, so that the underlying vector space $L$ is finite-dimensional. Let $\mathcal{O} \subseteq \ldef$ be some set of gauge orbits of the dgla. There is an ideal $\mathcal{I} \subset C^\infty(L^1)$ of smooth functions vanishing on $\mathcal{O}$. Our space of functions for the Poisson algebra will be the quotient ring
$$ \mathcal{R}_\mathcal{O} = C^\infty(M) / \mathcal{I}.$$
Given some function $f \in C^\infty(L^1)$, let $d_{DR} f \in \Gamma(T^\ast L^1)$ be the de Rham differential. For each $X \in L^1$, the pairing of the MCP structure gives an isomorphism $T_X^\ast L^1 = (L^1)^\vee \cong B^2$, which induces an isomorphism of trivial bundles $T^\ast L^1 \cong L^1 \times (L^1)^\vee \cong L^1 \times B^2$. Let $Df$ be the image of $d_{DR}$ under this map, so that $Df$ is a section of the trivial $B^2$ bundle over $L^1$. At each $X \in L^1$ we obtain the operator $\delta_X$ as the adjoint of $d_X= d+[X,\cdot]$, which is a differential if and only if $X \in \ldef $. We may then apply the operator $\delta_X$ to $Df$ at each $X \in L^1$ to obtain a section
$$ \delta_XD f \in \Gamma(L^1 \times B^1),$$
of the trivial $B^1$ bundle over $L^1$. 
\begin{lem} \label{lem:fdiff}
Let $[f]$ be a class in $\mathcal{R}_\mathcal{O}$ with representative function $f$. Then for each $X \in \mathcal{O}$, the value of 
$$ \delta_X D f \in B^1$$
does not depend on the choice of representative $f$.
\end{lem}
\begin{proof}
Let $g$ be a function in the ideal $\mathcal{I}$. Given $X= X(0) \in \ldef$ and some arbitrary element $V \in L^0$, let $X(t):[0,T) \to \ldef$ be the differentiable path in $L^1$ generated by gauge transformations, so that
$$\frac{dX(t)}{dt} = d_{X(t)} V,$$
since we are in finite-dimensions there is a $T>0$ such that the solution exists for $0<t<T$. Let $g(t)$ be the value of $g$ evaluated on the path. Since $g \in \mathcal{I}$, $g(t)=0$. We then have

$$0 = \left .\frac{dg}{dt} \right |_{t=0}= (d_{X(0)} V,  Dg(0))$$ 
where $(\cdot, \cdot)$ is the pairing for the MCP structure. Using the adjoint property we have
$$0 =( V, \delta_{X} Dg(0)).$$ 
But since $V$ is arbitrary and the pairing is non-degenerate, we must have 
$$\delta_{X} Dg=0$$
for all $g \in \mathcal{I}$. The result then follows.
\end{proof}
We can now define the Poisson algebra on $\mathcal{R}_\mathcal{O}$. Let $[f]$ and $[g]$ be two classes in $\mathcal{R}_\mathcal{O}$, with representatives $f$ and $g$. 
\begin{thm} \label{thm:poi}
The operation $\{\cdot, \cdot \} : \mathcal{R}_\mathcal{O} \times \mathcal{R}_\mathcal{O} \to \mathcal{R}_\mathcal{O}$ given by 
$$\{[f],[g]\} = (X, \delta_X Df \wedge \delta_X Dg)$$
gives $\mathcal{R}_\mathcal{O}$ the structure of a Poisson algebra.
\end{thm}
\begin{rem}
The operation $\delta_X Df \wedge \delta_X Dg$ is interpreted as a Lie bracket in Section~\ref{sec:cot}, viewed in this way one should note the similarity with the formula for the Lie-Poisson structure on coadjoint orbits of a Lie algebra.
\end{rem}
By Lemma~\ref{lem:fdiff} the bracket is a well-defined operation on $\mathcal{R}_\mathcal{O}$. The Leibniz identity, antisymmetry and linearity all follow trivially from the definition. To prove  that it defines a Poisson bracket we must then establish the Jacobi identity. To prove we first introduce some additional structure. Given an element $\beta \in B^2$ we define the operator
$$O_{X, \beta} : B^2 \to B^2$$
as
$$(A, O_{X,\beta} \alpha) = ([A, \rho_X \delta_X \beta], \alpha),$$
for all $A \in L^1$ and $\alpha \in B^2$. Now again, let $[f]$, $[g]$ be two elements of $\mathcal{R}_\mathcal{O}$ and let $f$, $g \in C^\infty(L^1)$ be representatives. Then $ (X, \delta_X  Df \wedge \delta_X Dg ) \in C^\infty(L^1)$ is a representative of the Poisson bracket $\{[f],[g]\}$. As with $f$ and $g$ we can form its de Rham differential $d_{DR}$ and use the isomorphism with $B^2$ to obtain a section of the trivial $B^2$ bundle over $L^1$,
$$D\{f, g\} \in \Gamma(L^1 \times B^2).$$
\begin{lem} \label{lem:dr}
The section $D\{f, g\} \in \Gamma(L^1 \times B^2)$ is given by
$$D \{f, g\} = \delta_X Df \wedge \delta_X Dg - O_{X,Dg} Df + O_{X,Df} Dg -{D^2f}( \nu_{X} Dg, \cdot) + {D^2g}( \nu_{X} Df, \cdot) ,$$
where $D^2f \in \Gamma(L^1 \times (B^2 \otimes B^2))$ is the symmetric tensor field obtained by the second coordinate derivative (Hessian) of $f$, mapped to $B^2$.
\end{lem}
\begin{proof}
Let $X(t):[0,T) \to L^1$, $X(0)=X$ be a smooth path in $L^1$, with $\dot X = dX/dt|_{t=0}$. Then take the 1-form $\alpha$ on $L^1$ given by the de Rham differential $d_{DR}(X, \llbracket Df, Dg \rrbracket)$. Let $\alpha_X \in (L^1)^\vee$ denote the value of $\alpha$ at $X$, then we have
$$\left .\frac{d}{d t}(X,\delta_X Df \wedge \delta_X Dg) \right |_{t=0}   =   \alpha_{X} (\dot X) = (\dot X, D\{f, g\})$$
and since $\dot X$ is arbitrary, this determines $D\{f, g\}$. Now, again using dot to denote $t$ derivative we have
\begin{align*}& \left .\frac{d}{d t}(X,\delta_X Df \wedge \delta_X Dg) \right |_{t=0}  = (\dot X, \delta_X Df \wedge \delta_X Dg) + (X, \dot{\delta_X} Df \wedge \delta_X Dg) \\ &+( X, \delta_X \dot{Df} \wedge \delta_X Dg)+( X, \delta_X {Df} \wedge \dot{\delta_X} Dg)+( X, \delta_X {Df} \wedge \delta_X \dot{Dg}).
\end{align*}
The first term is immediate, we now look at the second term, we find
\begin{align*}(X, \dot{\delta_X} Df \wedge \delta_X Dg) &= -(X, \delta_X Dg \wedge \dot{\delta_X} Df) = -(\rho_X \delta_X D g, \dot \delta_X f) \\
&=- (\dot{d_X} \rho_X \delta_X Dg, Df) = -([\dot X,  \rho_X \delta_X Dg], Df) \\
& = (\dot X, -O_{X, Dg} Df).
\end{align*}
The fourth term follows analogously. Next we examine the third term. The second coordinate derivative of $f$ will lead to $D^2f$, a symmetric tensor field on $L^1$, $D^2f \in \Gamma(L^1 \times (B_2 \otimes B_2))$. Then we have $\dot Df = D^2f(\dot X, \cdot)$ and we obtain 
\begin{align*}(X, {\delta_X} \dot{Df} \wedge \delta_X Dg) &= (X, {\delta_X} {D^2f}(\dot X, \cdot) \wedge \delta_X Dg)  =-(X, \delta_X Dg\wedge {\delta_X} {D^2f}(\dot X, \cdot) ) \\
&=-(d_X \rho_X \delta_X dDg, {D^2f}(\dot X, \cdot) )=-(\nu_{X} Dg, {D^2f}(\dot X, \cdot) )  \\
& =-(\dot X, {D^2f}( \nu_{X} Dg, \cdot) ) \end{align*}
where the last line follows from the symmetry of $D^2f$. 
\end{proof}
\begin{lem} \label{lem:jac}
The bracket of Theorem~\ref{thm:poi} satisfies the Jacobi identity.
\end{lem}
\begin{proof}
Let $[f]$, $[g]$, $[h]$ be three elements in $\mathcal{R}_\mathcal{O}$ and fix three representatives $f$, $g$, $h \in C^\infty(L^1)$. 
Using Lemma~\ref{lem:dr} we have
\begin{align*} &\{f, \{g,h\}\}+ \{g, \{h,f\}\} + \{h, \{f,g\}\} = \\ \nonumber &(X,  \delta_X Df \wedge \delta_X( Dg, \wedge \delta_X Dh ))+ \circlearrowright\\ \nonumber
 +&(X,  \delta_X Df \wedge \delta_X( O_{X, Dg}Dh  - O_{X, Dh}Dg))+ \circlearrowright \\\nonumber
+ & (X, \delta_X f \wedge \delta_X (D^2 g(\nu_{X} Df , \cdot) - D^2 f(\nu_{X} Dg , \cdot) ))+\circlearrowright,
\end{align*}
where $\circlearrowright$ denotes cyclic permutations. We need to show that this is zero for $X \in \mathcal{O} \subseteq \ldef$. Now using the BV algebra identities we can write the second row above as
$$(X,  \delta_X Df \wedge \delta_X( \delta_X Dg \wedge \delta_X Dh ))+ \circlearrowright = (X, \delta_X( \delta_X Df \wedge \delta_X Dg \wedge \delta_X Dh )),$$
which vanishes as $d_X X=0$ by property 1 of MCP structures. For the third row we may write 
$$(X,  \delta_X Df \wedge \delta_X( O_{X, Dg}Dh)) = (\nu_{X} Df , O_{X, Dg}Dh),$$
which by using the definition of the $O$ operator, and including another term from the cyclic permutations we find
$$-([(\nu_{X} Df , \rho_X \delta_X Dh], Dg)+ ([\nu_{X} Dh , \rho_X \delta_X Df], Dg).$$
Now recalling that $\nu_{X} = d_X \rho_X \delta_X$ we obtain
$$- (d_X[(\rho_X \delta_X Df , \rho_X \delta_X Dh], Dg),$$
using property (3) of the MCP structure, this is
$$-(k_X \rho_X[ \delta_X Df , \rho_X  Dh]_\mathcal{B}, \delta_X Dg),$$
where $[\cdot, \cdot]_\mathcal{B}$ is the Gerstenhaber bracket of the BV algebra on $B$. Using the definition of the bracket this gives
$$(k_X X,\delta_X(  \delta_X Dh \wedge\delta_X  Df) \wedge \delta_X Dg),$$
adding in cyclic permutations yields the same identity as for the second row, which vanishes using $d_XX=0$. Finally we need to study the fourth row, which yields
\begin{align*} &(X, \delta_X f \wedge \delta_X (D^2 g(\nu_{X} Df , \cdot) - D^2 f(\nu_{X} Dg , \cdot) )) \\&=   D^2 g(\nu_{X} Df , \nu_{X} Dh) - D^2 f(\nu_{X} Dg , \nu_{X} Dh) +\circlearrowright,
\end{align*}
which vanishes by the symmetries of $D^2$. Finally, note that, evaluated on $\mathcal{O}$, the expressions do not depend on the choice of representatives $f$, $g$, and $h$, and the result is established.
\end{proof}
This completes the proof of Theorem~\ref{thm:poi}.

\subsection{Compatible sub-algebra}

Suppose now we have a subalgebra $\mathcal{L}^\prime$ compatible with an MCP structure ($\mathcal{L}, \mathcal{B}, (\cdot, \cdot))$. We obtain a Poisson bracket on $\ldefp$ in this case also. This time, we consider the ideal $\mathcal{I}^\prime$ as smooth functions vanishing on some set of gauge orbits $\mathcal{O}^\prime$ of $\ldefp \subset \ldef$. A gauge orbit of $\mathcal{L}^\prime$ is, in general, a subset of a gauge orbit of $\mathcal{L}$), so let $\mathcal{O}$ be the smallest (in the sense of inclusions) set of gauge orbits of $\ldef$ containing ${\rm MC}(\mathcal{L}^\prime)$. Now form the quotient ring
$$\mathcal{R}_\mathcal{O}^\prime = C^\infty(L^1)/\mathcal{I}^\prime.$$
In this case Lemma~\ref{lem:fdiff} does not hold, however we still have a Poisson bracket. Let $[f]$ and $[g]$ be two functions in $\mathcal{R}_\mathcal{O^\prime}$ with representatives $f$ and $g$. As before, at each $X \in L^1$ we can form the expressions $Df$ and $Dg$.
\begin{lem} \label{lem:comp}
The expression
$$(X, \delta_X Df \wedge \delta_X Dg)$$
depends only on the classes $[f]$ and $[g] \in \mathcal{R}_\mathcal{O^\prime}$ for $X \in \ldefp$. 
\end{lem}
\begin{proof}
Consider a function $h$ which vanishes on $\mathcal{O}^\prime$, but not $\mathcal{O}$. Then at $X$, $Dh \in N^2_X$ (recall the definition of orthogonal complement in Definition~\ref{def:comp}). Then evaluating we have
$$(X, \delta_X Dh \wedge\alpha ) = (\nu_X Dh, \alpha)=0$$
for all $\alpha \in B^2$, by the exact sequence requirement of compatibility.
\end{proof}
\begin{prop}
The expression in Lemma~\ref{lem:comp} defines a Poisson bracket on $\mathcal{R}_\mathcal{O}^\prime$.
\end{prop}
\begin{proof}
By Lemma~\ref{lem:comp} the expression is well-defined. Picking any three representatives $f$, $g$, and $h$, Lemma~\ref{lem:jac} shows that the operation satisfies the Jacobi identity.
\end{proof}

\subsection{Vector bundle case} \label{sec:vb}

Two of our examples, deformations of Poisson and symplectic structures, correspond to infinite-dimensional dglas. To deal with these we expand our scope to dglas where each vector space $L^i$ is the space of sections of some smooth vector bundle over a smooth manifold $M$, $L^i = \Gamma(E^i)$, and $B^i$ is the space of sections of the dual bundle corresponding to $L^{i-1}$, $B^i = \Gamma((E^i)^\ast)$. This section is essentially a summary of the arguments in Ref.~\cite{machon2020poisson}, which go into more detail. Fix a volume form $\mu$, then for the pairing we take
$$(\cdot, \cdot): L^i \times B^{i+1} \to k, \quad  (V,\alpha) \mapsto \int_M ((V, \alpha)) \mu,$$
where $ ((\cdot, \cdot)) $ is the natural pairing between each fiber of $L^i$ and $B^{i+1}$. 
\begin{rem}
We will actually take this pairing with a factor of $(-1)^i$, to ensure the signs work out, but does not affect the construction. 
\end{rem}
Additionally we require that the product $\wedge$ and homomorphisms $\rho_X$ act fiberwise, and that the differentials, both $d_X$ and $\delta_X$, are differential operators. We can then define our space $\mathcal{C}$ of functions as `local functionals' (see for instance~\cite{costello2021factorization}, section 3.5.1), which can be written as an integral
$$F(X) = \int_M f( j_r X) \mu,$$
of some smooth function $f: J_r (E(L^1) ) \to k$, for finite $r \geq 0$, where $E(L^1)$ is the total space of the bundle for which $L^1$ is the space of sections. Once again, we can define the ideal $\mathcal{I}$ as those functions in $\mathcal{C}$ which restrict to zero on $\ldef$ (or a subset of gauge orbits $\mathcal{O}$), and consider the quotient
$$\mathcal{R}_\mathcal{O} = \mathcal{C}/ \mathcal{I}.$$
For some one-parameter family $X(t):[0,T) \to L^1$, with $X(0)=0$ and $dX(t)/dt |_{t=0} = \dot X$ we can define the differential $DF$ in the direction of the function in $\mathcal{C}$ as the element of $B^2$ defined as
$$
\left .\frac{d \tilde F}{dt} \right |_{t=0} =  (  \dot{X}, DF )_\mu,
$$
The analogue of Lemma~\ref{lem:fdiff} then follows identically, with the caveat that we assume that for all $\lambda \in L^0$, there is an $\epsilon \neq 0$ such that the gauge transformation by $\epsilon \lambda$ (recall the definition in Section~\ref{sec:prelim}) exists. Now let $[F]$, $[G]$ be two elements of $\mathcal{R}_\mathcal{O}$ with $F$, $G$ representatives, then we can define the bracket operation
$$\{\cdot, \cdot\} : \mathcal{R}_\mathcal{O} \times \mathcal{R}_\mathcal{O} \to \mathcal{R}_\mathcal{O},$$
given by
$$\{F, G\} = \int_M ((X, \delta_X DF \wedge \delta_X DG)) \mu.$$
\begin{thm}
The bracket operation turns $\mathcal{R}_\mathcal{O}$ into a Poisson algebra.
\end{thm}
The proof proceeds identically to Theorem~\ref{thm:poi}, (see for example~\cite{morrison1998hamiltonian} and \cite{hamilton1982inverse} for discussions of relevant infinite-dimensional calculus). The special case of Poisson structures is dealt with in Ref.~\cite{machon2020poisson}, and {\em mutatis mutandis} may be adapted directly to the case here.

\section{Hamiltonian flows of Maurer-Cartan elements} \label{sec:flow}

An MCP structure defines a notion of hamiltonian flow for Maurer-Cartan elements. As before let $\mathcal{O}$ be some set of gauge orbits, and consider an element $[f]$ in the ring $\mathcal{R}_\mathcal{O}$, represented by the function $f$. Then as in Section~\ref{sec:POI}, let $Df$ be the de Rham differential of $f$, thought of as a section of the trivial $B^2$ bundle over $L^1$.
\begin{prop} \label{prop:flow}
Let $X \in \mathcal{O}$ be a Maurer-Cartan element, then the hamiltonian flow of the MCP structure due to $[f]$ is given by
$$ \frac{dX}{dt}= \rho_{X,c} Df = d_X \rho_X \delta_X Df = d \rho_X \delta_X Df +[X, d \rho_X \delta_X Df],$$
where $d$ is the differential of the base dgla.
\end{prop}
Observe that the hamiltonian flow acts by gauge transformations.
\begin{rem}
We will not discuss questions of existence for the hamiltonian flows, though we note that finite-time singularities can occur for the flow of Poisson structures defined in Section~\ref{sec:poisson}, even for `linear functionals', see Ref.~\cite{machon2020poisson}, which can be seen by taking a Poisson structure on $T^3$.
\end{rem}

\begin{rem}
For a general Poisson Lie algebroid, one can attempt to integrate it into a Lie groupoid, a problem studied by Crainic and Fernandes~\cite{crainic2004integrability}. One could also attempt to integrate an MCP algebroid. This should yield a structure analogous to the Degline groupoid, where instead of general gauge transformations one considers only `hamiltonian' gauge transformations that arise as hamiltonian flows of Maurer-Cartan elements
\end{rem}

\begin{lem}
The hamiltonian flow of MC elements due to $[f]$ preserves the value of $[f]$.
\end{lem}
\begin{proof}
$$
\frac{d [f]}{ dt} = (\dot X, Df) = ( \rho_{X,c} Df , Df) = (X, \delta_X Df \wedge \delta_X Df)=0.
$$
\end{proof}
\begin{exam}
On a symplectic manifold the flow equation is the pde
$$ \partial_t \omega = d d^\Lambda \beta,$$
where $\beta$ is the 2-form giving the derivative of $f$ (see Section~\ref{sec:cv}).
\end{exam}

A Casimir of a Poisson algebra is an element $C$ in the center of the algebra. The following is immediate.
\begin{lem}
Any gauge-invariant function is a Casimir of the Poisson bracket for an MCP structure.
\end{lem}

\begin{exam} 
(\cite{machon2020poisson}, Example 4.)
On a closed 3-manifold $M$, a regular Poisson structure has a Godbillon-Vey invariant $GV \in H^3(M, \mathbb{R}) \cong \mathbb{R}$, which is diffeomorphism invariant, hence gauge invariant. 
\end{exam}
\begin{rem}
On each gauge orbit the Poisson structure induced by MCP may or may not be symplectic (see Section~\ref{sec:alg}), in the latter case it is in principle possible to construct a Casimir which is not gauge invariant.
\end{rem}

\section{Lie algebroids on gauge orbits}~\label{sec:alg}

In this section we restrict ourselves to the finite-dimensional setting. Consider a gauge orbit $\mathcal{O} \subset \ldef$. Since $\mathcal{O}$ arises from the action of the Lie algebra $L^0$ on $L^1$, it follows that 
$\mathcal{O}$ is a smooth immersed submanifold of $L^1$. Now the bracket in Theorem~\ref{thm:poi} defines a smooth anti-symmetric 2-vector field $\Pi$ on $L^1$ which is integrable on Maurer-Cartan elements, i.e. satisfies $[\Pi, \Pi]_S=0$ on $\ldef$ (with $[\cdot , \cdot]_S$ the Schouten-Nijenhuis bracket)

\begin{prop}
The tensor $\Pi$ defines a Poisson structure on each gauge orbit $\mathcal{O}$.
\end{prop}
\begin{proof}
At some $X \in \mathcal{O}$, $\Pi(a, \cdot)=0$ for all $a$ in the conormal space to $\mathcal{O}$ at $X$, it follows that the tensor $\Pi$ is well-defined at each $T^\ast_X \mathcal{O}$. Since $\mathcal{O}$ is an immersed submanifold $\Pi$ is smooth on $\mathcal{O}$. Both of these together imply that integrability extends to the tensor $\Pi$ on $\mathcal{O}$.
\end{proof}
Recall the definition of Lie-Poisson algebroids in Section~\ref{sec:prelim}.
\begin{cor}
Each gauge orbit has a Lie-Poisson algebroid $\mathcal{O}$, $A_\mathcal{O}$.
\end{cor}
We now discuss of the structure of this algebroid. At each point $X \in \ldef$, the MCP structure gives the following pair of complexes, connected by the homomorphism $\rho_X$
\begin{center}
\begin{tikzcd}
 \ldots  \arrow[r, "d_X"] & L^0 \arrow[r, "d_X"] & L^1  \arrow[r, "d_X"] & L^2  \arrow[r, "d_X"] & \ldots \\
 \ldots &  \arrow[l, "\delta_X"] B^{1} \arrow[u ,"\rho_X"]  & \arrow["\delta_X" above, l ] B^2  &  \arrow[l, "\delta_X" above] B^3 &  \arrow[l, "\delta_X" above] \ldots
 \end{tikzcd}
 \end{center}
We will first describe how this can be used to give some geometric structure to each gauge orbit. Given $X \in \ldef$ we have the BV algebra $\mathcal{B}_X$, let $Z(\mathcal{B}_X)$ denote the $\delta_X$-closed elements of $B$, with $Z^i(\mathcal{B}_X)$ the subspace with grading $i$.
\begin{lem}
At each point $X \in \mathcal{O}$, $T_X \mathcal{O} \cong d_X L^0 \subset L^1$ and $T^\ast_X \mathcal{O} \cong  B^2/Z^2(\mathcal{B}_X)$.
\end{lem}
\begin{proof}
At each $X \in \mathcal{O}$, the image of $L^0$ in $L^1$ is surjective on $T_X \mathcal{O}$ by construction, hence we can identify the tangent space of $\mathcal{O}$ as given by $T_X \mathcal{O} \cong d_X L^0 \subset L^1$. Using the non-degenerate pairing we can identity the dual as $B^2/N_X$, where $N_X$ is the conormal space of $\mathcal{O}$ at $X$, defined as elements $p \in B^2$ satisfying
$$(d_X \cdot,  p)=0.$$
Using the adjoint property and non-degeneracy of the pairing we see that $N_X= Z^2(\mathcal{B}_X)$.
\end{proof}

The orbit $\mathcal{O}$ possesses a trivial graded vector bundle with fiber $L$. There is a subbundle with fiber $L^0$, the gauge bundle, denoted $E_{L^0} \cong \mathcal{O} \times L^1$. At each $X \in \mathcal{O}$, there is a homomorphism of the fiber into the tangent space, $d_X : L^1 \to T_X \mathcal{O}$ such that the induced map on sections
$$ d_\mathcal{O}: \Gamma(E_{L^0}) \to \Gamma(T \mathcal{O})$$ 
is a surjective homomorphism of Lie algebras. We also have a trivial graded bundle with fiber $B$ over $\mathcal{O}$, and a subbundle $E_{B^1}$ with fiber $B^1$. At each $X$ we have the differential $\delta_X$ of the BV algebra, which induces a homomorphism $ \delta_X: T^\ast_X \mathcal{O} \to B^1$. The collection of all $\delta_X$ for each $X \in \mathcal{O}$ then yields a homomorphism of sections
$$ \delta_\mathcal{O}: \Gamma(T^\ast \mathcal{O}) \to \Gamma(E_{B^1}) $$
including the homormorphisms $\rho_X : B^1 \to L^1$ for each $X$ gives us the diagram
\begin{center}
\begin{tikzcd}
  \Gamma(E_{L^0})  \arrow[r, "d_\mathcal{O}"] & \Gamma(T\mathcal{O})  \\
 \Gamma(E_{B^1}) \arrow[u ,"\rho_\mathcal{O}"]  & \arrow["\delta_\mathcal{O}" above, l ] \Gamma(T^\ast\mathcal{O})   \end{tikzcd}
 \end{center}
 The composition gives the homomorphism $\nu_\mathcal{O} : \Gamma(T^\ast \mathcal{O}) \to \Gamma(T \mathcal{O})$. 
\begin{lem}
The homomorphism $\nu_\mathcal{O}$ is the anchor of the Poisson-Lie algebra on the gauge orbit $\mathcal{O}$.
 \end{lem}
 \begin{proof}
The anchor of a Lie-Poisson structure on a manifold $M$ is the map induced by the Poisson structure $\Pi^\# : \Gamma(T^\ast M) \to \Gamma(TM)$. Using the properties of the MCP algebroid we have, at each $X \in \mathcal{O}$
$$\Pi(a, b) = (X, \llbracket a, b \rrbracket) = (\nu_X a, b)$$
which implies the result.
 \end{proof}
 
 \begin{rem}
Everything said in this section holds for compatible sub-dglas also.
 \end{rem}

\section{The cotangent Lie algebras} \label{sec:cot}

Given a Lie algebroid over a base manifold $M$, at each point $x \in M$, the kernel of the anchor defines a Lie algebra, the isotropy Lie algebra at $x$. Correspondingly we expect isotropy Lie algebras for MCP structures, which should associate a Lie algebra to each MC element. We shall in fact see that MCP structures define a number of Lie algebras for each MC element, including the isotropy algebra associated to the Lie algebroid on gauge orbits. The definition of these Lie algebras does not depend on the system being finite-dimensional, and so we return to the general setting.

In Section~\ref{sec:alg} we showed that the cotangent space of a gauge orbit at $X$ was $B^2 / Z^2(\mathcal{B}_X)$. Our first goal is to show that for a general BV algebra $\mathcal{B}$ (not necessarily associated to an MCP structure) there is a graded Lie algebra on $B / Z(\mathcal{B})$. In the analogy with Lie-Poisson structures, this corresponds to a Lie algebra on the cotangent space. 

\subsection{A graded Lie algebra}
Let $\mathcal{B} = (B, \delta, \wedge)$ be a BV algebra (not necessarily associated to an MCP structure), let $Z(\mathcal{B}) = {\rm ker} \; \delta$, the $\delta$-closed elements of $B$ and consider the quotient space $B / Z(\mathcal{B})$. Define the bracket $$\llbracket \cdot , \cdot \rrbracket : B / Z(\mathcal{B}) \times B / Z(\mathcal{B}) \to B$$
as
$$\llbracket a , b \rrbracket = (-1)^{|a|}(\delta a  ) \wedge( \delta b ) .$$
Composing with the natural projection $\pi: B \to B / Z(\mathcal{B})$ gives a bracket $$\llbracket \cdot , \cdot \rrbracket_\pi: B / Z(\mathcal{B}) \times B / Z(\mathcal{B}) \to B/ Z(\mathcal{B}).$$

\begin{prop} \label{prop:gla}
The bracket $$\llbracket \cdot , \cdot \rrbracket_\pi $$ defines a graded Lie algebra on $(\mathcal{B} / Z(\mathcal{B}))[-2]$.
\end{prop}
\begin{proof}
First observe that the bracket is well-defined modulo $\delta$-closed elements of $B^2$, and that it satisfies the appropriate graded skew-symmetry. Now, let $a$, $b$, $c$ be elements of $B$, and for brevity denote their degrees by $a$, $b$, $c$ also. Then the nilpotency of order 2 condition reads
\begin{multline*} (-1)^{b+ac}\delta( \delta a\wedge \delta  b \wedge \delta c) \\
+(-1)^{c+a+c b} \delta c \wedge \delta( \delta a \wedge  \delta b) +   (-1)^{a+ b+ac}\delta a \wedge \delta (\delta b \wedge \delta c) + 
(-1)^{b+ c+a b} \delta b \wedge \delta( \delta c \wedge \delta a) = 0.
\end{multline*}
or
$$ (-1)^{b+ac}\delta( \delta a\wedge \delta  b \wedge \delta c) +(-1)^{a c} \llbracket  a , \llbracket b ,c \rrbracket_\pi  \rrbracket_\pi +(-1)^{b a} \llbracket  b , \llbracket c ,a \rrbracket_\pi  \rrbracket_\pi + (-1)^{c b}\llbracket  c , \llbracket a ,b \rrbracket_\pi  \rrbracket_\pi =0$$
Now since $\delta^2=0$, $ \delta( \delta a\wedge \delta  b \wedge \delta c) \in Z(\mathcal{B})$ and hence the Jacobi identity holds modulo elements of $Z(\mathcal{B})$. Finally, note that a shift in grading by 2 does not change any signs.
\end{proof}
\begin{rem}
Note that neither $\delta$ nor the product $\wedge$ interact well with this graded Lie algebra. 
\end{rem}
Let $H_\bullet(\mathcal{B})$ denote the homology of the BV algebra (or any subspace), considered as a graded abelian Lie algebra.
\begin{lem}
There is a central extension of graded Lie algebras with bracket $\llbracket \cdot, \cdot \rrbracket_\pi$
\begin{center}
\begin{tikzcd}
H_\bullet(\mathcal{\mathcal{B}})[-2] \arrow[hookrightarrow, r] &C  \arrow[r, twoheadrightarrow, ] &(\mathcal{B} / Z(\mathcal{B}))[-2]
\end{tikzcd}
\end{center}
\end{lem}
\begin{proof}
This follows as the Jacobi identity holds up to a $\delta$-exact term.
\end{proof}
The purpose of the `extension by homology' is, roughly, to allow one to consider an MCP structure on the entire set $\ldef$, rather than just a single orbit. Correspondingly the cotangent space is extended by elements of homology which are, roughly, dual to cohomology classes in the dgla $\mathcal{L}$, whose elements give deformations transverse to the gauge orbits (in the unobstructed case).

\subsection{The cotangent Lie algebras for an MCP structure}
Recall the definition of the $O$ operators from Lemma~\ref{lem:dr}.
\begin{prop} \label{prop:algebras}
Given an MCP structure and an MC element $X$ on a gauge orbit $\mathcal{O}$, there are three interrelated Lie algebras.
\begin{enumerate}
\item A graded Lie algebra on $B / Z(\mathcal{B}_X) [-2]$ with bracket
$$ \llbracket a, b \rrbracket_\pi = (-1)^{|a|} \delta_X a \wedge \delta_X b$$
whose restriction to $B^2/Z^2(\mathcal{B}_X)$ gives a Lie algebra $\mathfrak{c}_X$ on each cotangent space $T^\ast_X \mathcal{O} \cong B^2 / Z^2(\mathcal{B}_X)$.
\item A subalgebra $\mathfrak{h}_X \subset \mathfrak{c}_X$ given by the kernel of 
$$\nu_X = d_X \rho_X \delta_X : B^2 \to L^1.$$
\item The isotropy Lie algebra $\mathfrak{j}_X$ of the Lie algebroid, also on ${\rm ker}\; \nu_X$, with bracket
$$ [a, b]_{\mathfrak{j}_X} = \llbracket a, b \rrbracket_\pi + O_{X,a} b - O_{X,b}a.$$
\end{enumerate}
\end{prop}
\begin{rem}
In the case of a compatible sub-dgla, only the Lie algebra $\mathfrak{j}_X$ survives.
\end{rem}
\begin{proof}
That 1.~defines a Lie algebra is just Proposition~\ref{prop:gla}. For 2.~ observe that using the definition of MCP structures we see 
$$ \nu_{ X} \llbracket a, b \rrbracket_\pi = k_X d_X [\rho_{X} \delta_X a, \rho_{X} \delta_X b] = k_X [\nu_X  a, \rho_{X} \delta_X b]+k_X[\rho_{X} \delta_X a, \nu_X  b]=0.$$
Hence $\mathfrak{h}_X$ is a subalgebra of $\mathfrak{c}_X$. Finally, for the formula for $\mathfrak{j}_X$ we expand the definition of the Lie bracket $[\cdot, \cdot]_{\mathcal{A}_\mathcal{O}}$ on the algebroid $A_\mathcal{O}$ (see Section~\ref{sec:prelim}). We then obtain
$$[a, b]_{\mathcal{A}_\mathcal{O}} = d_{DR} \Pi(a,b) + \iota_{\nu_\mathcal{O}a} db - \iota_{\nu_\mathcal{O} b} da$$
 where $d_{DR}$ is the de Rham differential on $\mathcal{O}$. Now assuming both $a$ and $b$ are in the kernel of $\nu_\mathcal{O}$ we have
 $$[a, b]_{\mathcal{A}_\mathcal{O}} = d_{DR} \Pi(a,b). $$
We can evaluate this by using the formula for the de Rham differential in Lemma~\ref{lem:dr}, which yields
$$ d_{DR} \Pi(a,b) = \llbracket a, b \rrbracket + O_{X, a} b-O_{X,b}a,$$
note that the second derivative terms vanish. 
\end{proof}
\begin{rem}
$\mathfrak{c}_X$ and $\mathfrak{h}_X$ (and sometimes $\mathfrak{j}_X$) can be extended by the homology $H_2(\mathcal{B}_X)$. For example we may extend the cotangent Lie algebra by homology 
\begin{center}
\begin{tikzcd}
H_2(\mathcal{\mathcal{B}})[-2] \arrow[hook]{r} & \mathfrak{l} \arrow[twoheadrightarrow]{r} &  \mathfrak{c}.
\end{tikzcd}
\end{center}
In some cases algebras $\mathfrak{h}_X$ and $\mathfrak{j}_X$ can be extended to graded Lie algebras. 
\end{rem}

\section{The relation to Lie-Poisson structures} \label{sec:lp}

The formula for the Poisson bracket in Theorem~\ref{thm:poi} is strongly reminiscent of Lie-Poisson structure on the dual of a Lie algebra~\cite{kirillov2004lectures,marsden2013introduction}. In certain cases the bracket of Theorem~\ref{thm:poi} reduces to a Lie-Poisson bracket. In Remark~\ref{rem:types} we noted that the MCP structures we have found come in two flavors, those with a base dgla having trivial Lie bracket, and those where the base dgla has trivial differential. In the former case $\ldef$ is a vector space corresponding to $d$-closed elements of $L^1$, $Z(L^1)$, and for each $X \in \ldef$, the tangent dlga $\mathcal{L}_X=\mathcal{L}$. It follows that there is a single BV algebra $\mathcal{B}$. Now recalling the definition of the cotangent Lie algebra $\mathfrak{c}$ from Section~\ref{sec:cot}, let $\mathfrak{l}$ be the extension of $\mathfrak{c}$ by the second homology group of $\mathcal{B}$, so that we have the short exact sequence
\begin{center}
\begin{tikzcd}
H_2(\mathcal{\mathcal{B}})[-2] \arrow[hook]{r} & \mathfrak{l} \arrow[twoheadrightarrow]{r} &  \mathfrak{c}.
\end{tikzcd}
\end{center}
In finite-dimensions, the cotangent algebra $\mathfrak{l}$ then becomes a Lie algebra on the dual of $Z(L^1)$, and we find the following.
\begin{prop}
A finite-dimensional MCP structure $(\mathcal{L}, \mathcal{B}, (\cdot, \cdot))$ whose base dgla has trivial Lie bracket is equivalent to a Lie-Poisson structure on the dual of the extension of the cotangent algebra, $\mathfrak{l}$.
\end{prop}
Observe also that the $O$ operators are trivial in this case and we have
\begin{lem}
A finite-dimensional MCP structure $(\mathcal{L}, \mathcal{B}, (\cdot, \cdot))$ whose base dgla has trivial Lie bracket, the Lie algebras $\mathfrak{h}_X$ and $\mathfrak{j}_X$ coincide.
\end{lem}
Any BV algebra (with an appropriate dual space) determines such an MCP structure simply by taking the dual space with the adjoint differential and trivial Lie bracket.

\begin{rem}
In the infinite-dimensional case of symplectic structures, one also has an infinite-dimensional Lie-Poisson MCP structure, reminiscent of hamiltonian formulations hydrodynamics (see e.g.~\cite{arnold2021topological, marsden2013introduction}).
\end{rem}

\section{Commutative Frobenius algebras} \label{sec:frobenius}

Let $A$ be a finite-dimensional commutative Frobenius algebra. Finite-dimensionality is chosen here for convenience, the construction should work for well-behaved infinite-dimensional $A$.

\subsection{The Hodge decompisition of Hochschild Cohomology}

This subsection holds for a general associative algebra $A$. As shown by Gerstenhaber~\cite{gerstenhaber1964deformation,gerstenhaber1963cohomology}, the Hochschild cohomology of $A$, ${\rm HH}^\bullet(A,A)$ is a Gerstenhaber algebra, it has a graded commutative cup product $\cup: {\rm HH}^i(A,A) \times {\rm HH}^j(A,A) \to {\rm HH}^{i+j}(A,A)$ and Gerstenhaber bracket $[\cdot, \cdot]: {\rm HH}^i(A,A) \times {\rm HH}^j(A,A) \to {\rm HH}^{i+j-1}(A,A)$. In particular, the Gerstenhaber bracket turns ${\rm HH}^\bullet(A,A) [-1]$ into a graded Lie algebra. The Hochschild cohomology groups (and cochains, and homology/chains) have a Hodge decomposition given by Gerstenhaber and Schack~\cite{gerstenhaber1987hodge} (see also Ref.\cite{quillen1970co}, Theorem 8.6 and Corollary 8.7)
$$ {\rm HH}^p(A,A) = \bigoplus_{i=1}^{p} {\rm HH}^{i,p-i}(A,A),$$
where each summand is an eigenspace with eigenvalue $2^i-2$ of the shuffle operator. The behaviour of the Gerstenhaber bracket and cup product under the Hodge decomposition was studied by Bergeron and Wolfgang~\cite{bergeron1995decomposition}. If we write
$$ \mathcal{H}^i = \bigoplus_k {\rm HH}^{i,k-i}(A,A),$$
then Bergeron and Wolfgang prove that
\begin{equation} \label{eq:bw} [\mathcal{H}^i, \mathcal{H}^j ] \subset \bigoplus_{k \leq i+j-1} \mathcal{H}^k, \quad \mathcal{H}^i \cup \mathcal{H}^j \subset \bigoplus_{k \leq i+j} \mathcal{H}^{k}.
\end{equation}
Equivalently they define the ideals 
$$ \mathcal{F}_q = \bigoplus_{r \geq q} H^{\bullet, r}(A,A)$$
which satisfy
$$ [\mathcal{F}_p,\mathcal{F}_q] \subset \mathcal{F}_{p+q}, \quad \mathcal{F}_p \cup \mathcal{F}_q \subset \mathcal{F}_{p+q}.$$
The following is then immediate.
\begin{prop}
For $q \geq 0$, the quotient $ \mathcal{G}_q = {\rm HH}^\bullet(A,A)/ \mathcal{F}_q$ has the structure of a Gerstenhaber algebra.
\end{prop}
We will be interested in the algebra $\mathcal{G}_1$, which corresponds to Hochschild cohomology modulo all non-skewsymmetric classes. Let $X \in \mathcal{G}^i_1$ be an element of degree $i$.
\begin{lem}
There is a canonical representative cochain of $X$, $\tilde X \in C^i(A,A)$, which is skew-symmetric.
\end{lem}
\begin{proof}
The Hodge decomposition passes to the Hochschild cochains~\cite{gerstenhaber1987hodge}. Moreover, the Hochschild differential $d_{HH}: {C}^{i}(A,A) \to {C}^{i+1} (A,A)$ respects the Hodge decomposition, $d_{HH} : \mathcal{C}^i \to \mathcal{C}^i$, where
$$ \mathcal{C}^i =  \bigoplus_k {C}^{i,k-i}(A,A),$$
hence there is a unique representative for each class in ${\rm HH}^{i,0}(A,A) \cong \mathcal{G}_1^i$. $C^{i,0}(A,A) = (\bigwedge^i A^\vee) \otimes A$, which implies the alternating symmetry.
\end{proof}
Henceforth we will refer to both an element and its canonical representative by the same symbol (typically $X$). Note that, as a vector space, $\mathcal{G}_1$ is identified with ${\rm HH}^{\bullet,0}(A,A)$. We then have the following characterisation of elements of $\mathcal{G}_1$.
\begin{prop}{(\cite{gerstenhaber1991shuffle}, Theorem 3)} \label{prop:halp} Elements of $\mathcal{G}_1$ consist of the skew multiderivations. \end{prop}
Consider the Gerstenhaber bracket on $\mathcal{G}_1$. Given two elements $X \in {\rm HH}^{i,0}(A,A)$, $Y\in {\rm HH}^{j,0}(A,A)$, their Gerstenhaber bracket will be
$$[X,Y] \in \bigoplus_{k \leq i+j-1} {\rm HH}^{k, i+j-1 - k}(A,A),$$
all the terms with $ k <  i+j-1$ can be thought of as `error terms' (see Ref.~\cite{bergeron1995decomposition}) which we then neglect when passing to the complex $\mathcal{G}_1$. Indeed, let $\pi : {\rm HH}^k(A,A) \to {\rm HH}^{k,0}$ be the projection associated to the Hodge decomposition, then the bracket on $\mathcal{G}_1$ is given by
$$[\cdot, \cdot]_{\mathcal{G}_1} = \pi [\cdot, \cdot],$$
where $[\cdot, \cdot]$ is the Gerstenhaber bracket.

\subsection{The dgla}
Returning to the commutative case, we now construct the base dgla for the MCP structure. We set
$$ \mathcal{L} = (\mathcal{G}_1[-1], 0, [\cdot, \cdot]_{\mathcal{G}_1}).$$
The Maurer-Cartan set is given by the cone
$$ \ldef = \left \{ X \in \mathcal{G}_1^2 \; | \; [X, X]_{\mathcal{G}_1}= 0  \right \}.$$
We can give an interpretation of $\ldef$. 
\begin{prop}
Elements of $X \in \ldef$ correspond to Lie algebras on $A$ making $A$ a Poisson algebra.
\end{prop}
\begin{proof}
As $X$ is a cocycle, by Proposition~\ref{prop:halp} $X$ is a multiderivation. Then $X \in \ldef$ implies $[X,X]_{\mathcal{G}_1}=0$. In this case the bracket $\mathcal{G}_1$ will be the projection of the Gerstenhaber bracket on $X$ onto ${\rm HH}^{3,0}(A,A)$. This projection is achieved by the usual alternating map from a tensor algebra to the exterior algebra. Explicitly, using the formula for the Gerstenhaber bracket (along with the symmetry of $X$) we have
$$[X,X](a,b,c) = 2(X(X(a,b),c) - X(a,X(b,c))$$
and then
\begin{align*}
[X,X]_{\mathcal{G}_1}(a,b,c) =  \frac{2}{3}  ( X(X(a,b),c) - X(a,X(b,c)) +  \\ X(X(b,c),a) - X(b,X(c,a))+ X(X(c,a),b) - X(c,X(a,b) )
\\  = \frac{4}{3} (X(X(a,b),c) + X(X(b,a),c)+X(X(c,a),b) )
\end{align*}
to obtain the final line we use the antisymmetry of $X$. We see then that the final line vanishing corresponds to the Jacobi identity for $X$, hence $X \in \ldef$ defines a Lie algebra.
\end{proof}

\subsection{The commutative algebra $\mathcal{B}$}

To define the MCP structure we need to specify a graded commutative algebra $\mathcal{B}$. This will be done using Hochschild homology. Recall that for a finite-dimensional commutative Frobenius algebra we have (see e.g.~\cite{cartan2016homological}, Proposition 5.1, or \cite{zhu2014co}) 
$$ {\rm HH}^i(A,A)^\vee \cong {\rm HH}_i(A,A),$$
 with the pairing
\begin{equation} \label{eq:hhpairing} (\cdot , \cdot):  {\rm HH}^i(A,A) \times {\rm HH}_i(A,A) \to k,
\end{equation}
defined as follows. Given a cochain $F \in C^p(A,A)$
$$ F = \alpha_1 \otimes \ldots \otimes \alpha_p \otimes a,$$
where each $\alpha_i \in A^\vee$, and Latin letters denote elements of $A$. Then given a chain $f \in C_p(A,A)$ given by
$$ f = b \otimes f_1 \otimes \ldots \otimes f_p,$$
the pairing on chains is defined as
$$ (F,f) = \langle a, b \rangle \sum_{i=1}^p \alpha_i(f_i),$$
where $\langle \cdot, \cdot \rangle$ is the bilinear form of the Frobenius algebra, the pairing is extended using bilinearity. The adjoint of the Hochschild differential $d$ with respect to this pairing is the boundary operator $b$ in Hochschild homology, and we obtain a non-degenerate pairing on homology and cohomology. The Hochschild homology groups ${\rm HH}_i(A,A)$ have the structure of a graded commutative algebra through the shuffle product (see~\cite{loday2013cyclic}, Section 4.2). Recall that a $(p,q)$ shuffle is given by a permutation $\sigma$ in the symmetric group $S_{p+q}$ such that
$$ \sigma(1) < \sigma(2) < \ldots \sigma(p), \quad \sigma(p+1) < \sigma(p+2) < \ldots < \sigma(p+q),$$
i.e.~a permutation such that the first $p$ and last $q$ elements appear in the same order. Given a shuffle $\sigma$, ${\rm sgn}(\sigma) = \pm 1$ it its sign as a permutation. We can then define the shuffle product $\Sh$ on $C_\bullet(A, A)$ as follows
\begin{multline*}
(a \otimes f_1 \otimes \ldots \otimes f_p ) \Sh (b \otimes f_{p+1} \otimes \ldots \otimes f_{p+q}) = \\
\sum_{\sigma = (p,q) \, {\rm shuffle}} {\rm sgn}(\sigma)a b \otimes \sigma \left (  f_1 \otimes \ldots \otimes f_{p+q} \right) 
\end{multline*}
\begin{prop} (See Ref.\cite{loday2013cyclic} Corollary 4.2.7) The product $\Sh$ gives ${\rm HH}_\bullet(A,A)$ the structure of a graded commutative algebra.
\end{prop}
Hochschild homology also has a Hodge decomposition~\cite{gerstenhaber1987hodge}, given as
$$ {\rm HH}_p(A,A) = \bigoplus_{i=1}^{p} {\rm HH}_{i, p-i}(A,A),$$
defined analogously to the cohomology case. In constrast to the Gerstenhaber bracket and cup product, the shuffle product behaves well with respect to the Hodge decomposition~\cite{gerstenhaber1991shuffle}, 
$$\Sh: {\rm HH}^{i,k}(A,A) \times {\rm HH}^{j,l}(A,A) \to {\rm HH}^{i+j,k+l}(A,A).$$
In particular, 
$$B = \bigoplus_i {\rm HH}^{i,0}(A,A)$$
equipped with the shuffle product is a graded commutative algebra. In fact, $B$ is just the exterior algebra of $A$.
$$ B = A \otimes \bigwedge^\bullet A,$$
and the shuffle product is (up to normalising prefactors), the wedge product $\wedge$. Hence the pair
$$ \mathcal{B}= (B, \wedge)$$
defines a graded commutative algebra. The next ingredient in the MCP structure is the pairing.
\begin{lem}
The pairing \eqref{eq:hhpairing} induces a non-degenerate pairing
$$ (\cdot, \cdot): L^{i+1} \times B^i \to k.$$
\end{lem}
\begin{proof}
A short calculation shows that the pairing \eqref{eq:hhpairing} respects the Hodge decomposition. Using the non-degeneracy of the pairing we see then that the orthogonal complement to $B^i$ is then $${\rm HH}^i(A,A) \big / \bigoplus_{k < i} {\rm HH}^{k,i-k}(A,A). $$
Hence the pairing induces a pairing between $B^i$ and  ${\rm HH}^{i}(A,A)/ \bigoplus_{k < i} {\rm HH}^{k,i-k}(A,A) $, but this is just $\mathcal{G}_1^i = L^{i+1}$.
\end{proof}
\subsection{The BV algebra structure}
We now have our triple $(\mathcal{L}, \mathcal{B}, (\cdot, \cdot))$, we must show that we obtain a BV algebra. Given some $X \in \ldef$ we obtain a differential $d_X$ on $\mathcal{L}$. Given $F \in \mathcal{L}^i$, it has a canonical representative Hochschild cochain $F \in C^{i,0}(A,A)$. The differential $d_X F$ is then given by
$$ d_X F = [X, F]_{\mathcal{G}_1} = \pi [X,F],$$
where $[\cdot, \cdot]$ is the Gerstenhaber bracket on Hochschild cohomology and $\pi$ is the projection onto skew tensors.
\begin{lem}
The differential $d_X$ acting on a given $F \in L^p(A,A) $ (with $F$ also denoting is canonical representative cochain) is represented by the cochain given by
$$d_X F = \sum_{i=0}^{p+1} (-1)^i d_{X,i} F,$$
where
\begin{align*}(d_{X,0} F)(a_1, \ldots, a_{p+2}) &= [a_1 , F(a_2, \ldots, a_{p+2})]_X,\\
(d_{X,i} F)(a_1, \ldots, a_{p+2}) &= F(a_1, \ldots, [a_i , a_{i+1}]_X, \ldots, a_{p+2}),\\
 (d_{X,n} F)(a_1, \ldots, a_{n+1}) &= [F(a_1, \ldots, a_{p+1}) , a_{p+2}]_X
\end{align*}
and $[\cdot, \cdot]_X$ is the Lie bracket on $A$ induced by $X$.
\end{lem}
\begin{proof}
This follows from the standard expression for the Hochschild derivative (see~\cite{loday2013cyclic}, Section 1.5.1).\end{proof}
Observe that this expression is {\em not} necessarily the canonical representative of $d_X F$, we would need to compose with the projection operator $\pi$. Using the pairing we then obtain an adjoint differential $\delta_X$ on $B$. We need not use the canonical representative $d_X$, since any non-canonical piece is orthogonal to $B$. We can write a general element $B^i$ as a sum of primitive elements $f$ of the form
$$f= a \otimes (f_1 \wedge f_2 \wedge \ldots \wedge f_i).$$
\begin{lem} \label{lem:83}
The differential $\delta_X f \in B_{p-1}(A, A)$ is twice the Chevally-Eilenberg differential on $A \otimes \bigwedge^\bullet (A)$ (thought of as a Lie algebra with bracket $[\cdot, \cdot]_X$), with formula.
\begin{multline*}  \delta_X f =2  \sum_{i=1}^p (-1)^{i+1} [a,  f_i ]_X \otimes (f_1 \wedge \ldots \hat{f}_i \wedge \ldots \wedge f_p ) + \\
2a \otimes \sum_{i<j}(-1)^{i+j} [f_i, f_j] \wedge  (f_1 \wedge \ldots \hat{f}_i \ldots \wedge {\hat f}_j \wedge \ldots \wedge f_p,
\end{multline*}
where $\hat{f}_i$ denotes removal of element $i$.
\end{lem}
\begin{proof}
Given a general element of $C_p(A,A)$,  $a \otimes f_1 \otimes \ldots \otimes f_p$ the adjoint of $d_X$ is the standard Hochschild boundary operator, with the product defined by $X$ instead of the product of the algebra. 
\begin{multline*}   [a,  f_1 ]_X \otimes f_2 \otimes \ldots \otimes f_p \\
+\sum_{i=1}^{p-1}(-1)^ia \otimes (f_1 \otimes \ldots \otimes [f_i,  f_{i+1}]_X \otimes \ldots \otimes f_p \\
+  (-1)^{p} [f_p ,  a]_X \otimes f_1 \otimes \ldots  \otimes f_{p-1} .
\end{multline*}
Then passing to the wedge product just sums over all permutations of the $f_i$ with sign then yields the formula given.
\end{proof}
\begin{lem}
The differential $\delta_X$ gives the triple $\mathcal{B}_X= (B, \delta_X, \wedge)$ the structure of a BV algebra.
\end{lem} 
\begin{proof}
We need only verify the order 2 property. We can separate the differential $\delta_X = \delta_1+\delta_2$ into two pieces, corresponding to the two lines in Lemma~\ref{lem:83}. That the second piece $\delta_2$ satisfies the order 2 identity is from the standard BV structure on the Chevally-Eilenberg chain complex, see for example~\cite{koszul1985crochet}, page 261 and~\cite{lian1993new}, page 644,~\cite{kosmann1995exact}, Example 1.1 and Section 5. For the first piece, $\delta_1$, observe that the product on $A$ and the bracket $[\cdot, \cdot]_X$ obey the identity $a[b,c]_X = [a,b]_X c$, which implies that the order 2 identity is satisfied. 
\end{proof}
\subsection{The homomorphism}
The homomorphism $\rho_X$ is defined by
$$(X, f \wedge g) = (\rho_X f , g).$$
Given an element $f \in B^1$ of the form $a \otimes f_1$ we have the following formula.
\begin{lem}
$$\rho_X f = a [ f_1, \cdot]_X.$$
\end{lem}
\begin{proof}
In general we can write
$$ X = \sum_i ( \alpha_{i,1} \otimes \alpha_{i,2}- \alpha_{i,2} \otimes \alpha_{i,1}) \otimes c_i,$$
where $c_i \in A$ and $\alpha_{i,j} \in A^\vee$. Now consider $f = a \otimes f_1$ and $g = b \otimes f_2$, then we have
$$(X, f \wedge g) = \sum_i \langle c_i , a b \rangle ( \alpha_{i,1}(f_1)\alpha_{i,2}(g_1) -  \alpha_{i,2}(f_1)\alpha_{i,1}(g_1))$$
from which the formula follows (using the invariance of the inner product $\langle \cdot, \cdot \rangle$ coming from the Frobenius structure).
\end{proof}
\begin{prop}
The homomorphism $\rho_X: B^1 \to L^0$ satisfies
$$[\rho_X f, \rho_X g]_{\mathcal{G}_1} =  \frac{1}{2} \rho_X [f, g]_{\mathcal{B}_X}.$$
\end{prop}
\begin{proof}
The bracket on $B^1$ is given by
$$[f,g]_\mathcal{B} = -\delta_X(f \wedge g) - (\delta_X f ) \wedge g + (\delta_X g) f.$$
Again let $f$, $g$ be given by 
$$ f = a \otimes f_1, \quad g = b \otimes g_1.$$
Now we have
$$\delta_X f = 2[a,  f_1]_X $$
and
\begin{align*}\delta_X(f \wedge g) &= 2 [ab , f_1]_X  \otimes g_1 -2 [ab , g_1]_X  \otimes f_1 - 2 ab \otimes [f_1 , g_1]_X.
\end{align*}
Now because $X$ is a cocycle in Hochschild homology we have the relation $[ab, c]_X =  [a ,b]_X c$ which then allows us to express the bracket as
$$[f, g]_{\mathcal{B}_X} = 2 a b \otimes [f_1 , g_1]_X.$$
Now the map $\rho_X$ is given by $ \rho_X f = a [f_1,   \cdot]_X,$
we then find
$$ \rho_X [f, g]_{\mathcal{B}_X}  = 2 a b[ [f_1 , g_1]_X , \cdot]_X.$$
Now given two elements $\alpha$ and $\beta \in L^0$ (which are derivations), the bracket on $\mathcal{L}$ is just the Gerstenhaber bracket on cohomology and is given by
$$ [\alpha, \beta]_{\mathcal{G}_1} = \alpha (\beta(\cdot))  - \beta( \alpha(\cdot)).$$
Now we have
$$ [\rho_X f, \rho_X g]_{\mathcal{G}_1} = a[f_1, b[g_1, \cdot]_X]_X - b[g_1, a[f_1, \cdot]_X]_X$$
Using the fact that the bracket $[\cdot, \cdot]_X$ is a derivation of the product on $A$ this becomes
$$a[f_1, b]_X [g_1, \cdot]_X + a b [f_1, [g_1, \cdot]_X]_X - b[g_1, a]_X [f_1, \cdot]_X - a b [g_1, [f_1, \cdot]_X]_X$$
Now using the fact that $[ab, c]_X =  [a ,b]_X c$ one may show the first and third terms cancel, using the Jacobi identity we see
$$ a b [f_1, [g_1, \cdot]_X]_X  - a b [g_1, [f_1, \cdot]_X]_X = - ab[\cdot, [f_1, g_1]_X]_X = \frac{1}{2} \rho_X [f, g]_{\mathcal{B}_X}.$$
\end{proof}
\begin{rem}
The factor of 1/2 can be set to 1 by changing our conventions for the wedge product.
\end{rem}
\subsection{The MCP Structure}
Our construction then yields the following.
\begin{thm} \label{thm:frobenius}
The triple $(\mathcal{L}, \mathcal{B}, (\cdot, \cdot))$ defines an MCP structure (with $k_X=1/2$). 
\end{thm}
The MCP structure defines a tensor on $L^1 \cong {\rm HH}^{2,0}(A,A)$ which is Poisson on $\ldef$. It is clear from the formula that this tensor is polynomial (in fact cubic) in coordinates of $L^1$. Hence the tensor defines a Poisson bracket on the coordinate ring of $\ldef$. We have seen that elements of $\ldef$ correspond to Lie brackets making $A$ a Poisson algebra. $\ldef$ is a cone $\mathcal{P}(A)$ in ${\rm HH}^{2,0}(A,A)$ defined by the vanishing of a set of homogeneous quadratic polynomials (one for each component of $[X,X] \in {\rm HH}^{3,0}(A,A)$) and hence defines a similar set in ${\rm HH}^2(A,A)$, using the Hodge decomposition of Hochschild cohomology. We then obtain the two corollaries quoted in the introduction.
\begin{manualtheorem}{1.1}[]
The coordinate ring of $\mathcal{P}(V)$ is a Poisson algebra.
\end{manualtheorem}
\begin{manualtheorem}{1.2}[]
Each gauge orbit $\mathcal{O} \subset \mathcal{P}(V)$ is a Poisson manifold.
\end{manualtheorem}

\section{Chevally-Eilenberg Complex} \label{sec:examlie}

In Section~\ref{sec:lp} we discussed how any BV algebra gives rise to an MCP structure by considering the dual space as a dgla with trivial Lie bracket. A canonical example of a BV algebra is the Chevally-Eilenberg chain complex (see see~\cite{koszul1985crochet}, page 261,~\cite{lian1993new}, page 644, ~\cite{kosmann1995exact}, Example 1.1 and Section 5). Here we consider this example in detail. 
We first relate it to an appropriate problem deformation theory. Let $\mathfrak{g}$ be a finite-dimensional Lie algebra over $k$ and consider the space $M$ of all (one-dimensional) central extensions of $\mathfrak{g}$ with $\mathfrak{a}$ the one-dimensional abelian Lie algebra. Consider the Lie algebra $\mathfrak{g} \oplus \mathfrak{a}$, following Nijenhuis and Richardson~\cite{nijenhuis1967deformations} deformations are controlled by the Chevally-Eilenberg complex with coefficients in $\mathfrak{g}\oplus \mathfrak{a}$, given by
\begin{multline*}
\left ( \left ( \bigwedge\nolimits ^\bullet (\mathfrak{g} \oplus \mathfrak{a})^\vee \right ) \otimes ( \mathfrak{g} \oplus \mathfrak{a} ) \right) [-1] \\= \left ( \left ( \bigwedge\nolimits ^\bullet \mathfrak{g} ^\vee \right ) \otimes \mathfrak{g}   \right)  \oplus  \left ( \left ( \bigwedge\nolimits ^\bullet \mathfrak{a} ^\vee \right ) \otimes \mathfrak{g}   \right) \oplus \left  ( \bigwedge\nolimits ^\bullet \mathfrak{g} ^\vee \otimes \mathfrak{a} \right)\oplus \left (  \bigwedge\nolimits ^\bullet \mathfrak{a} ^\vee \otimes \mathfrak{a}\right ) [-1] ,
\end{multline*}
with the standard differential $d_\mathfrak{g}$ and Nijenhuis-Richardson Lie bracket. If we consider only deformations corresponding to central extensions of $\mathfrak{g}$ by $\mathfrak{a}$, then only terms in the third summand above are non-zero, and we may consider instead the dgla with underlying vector space
$$L =  \left ( \left ( \bigwedge\nolimits ^\bullet \mathfrak{g} ^\vee \right ) \otimes  \mathfrak{a} \right) [-1],$$
in this case we give the standard expression for the differential $d_\mathfrak{g} :L^i \to L^{i+1}$ as
$$ (d_\mathfrak{g} X)(\alpha_1, \ldots, \alpha_{n+1}) = \sum_{i<j} (-1)^{i+j} X([\alpha_{i}, \alpha_j], \alpha_1, \ldots, \hat{\alpha}_i, \ldots, \hat{\alpha}_j, \ldots, \alpha_{n+1}),$$
where $\hat{\alpha}_i$ denotes removal of $\alpha_i$. The Nijenhuis-Richardson bracket on $L$ is trivial, which allows us to give the set of Maurer-Cartan elements as a vector space
$$ {\rm MC}(\mathcal{L}) = {\rm ker} \;d_\mathfrak{g} \cap \bigwedge \nolimits^2 \mathfrak{g}^\vee = Z^2(\mathfrak{g}),$$
accounting for automorphisms gives the standard result that one-dimensional central extensions are given by the cohomology group $H^2(\mathfrak{g})$. The underlying vector space for the BV algebra is the Chevally-Eilenberg chain complex
$$ B =  \bigwedge^\bullet \mathfrak{g},$$
with associative product given by the wedge product. The differential $\delta_\mathfrak{g}: B^i \to B^{i-1}$ is the adjoint of $d_\mathfrak{g}$ with respect to the natural pairing between $\bigwedge^n \mathfrak{g}^\vee$ and $\bigwedge^n \mathfrak{g}$. The differential $\delta_\mathfrak{g}$ satisfies the order 2 nilpotency condition making $(B, \wedge, \delta_\mathfrak{g})$ a BV algebra (see~\cite{koszul1985crochet}, page 261 and~\cite{lian1993new}, page 644).

Let $\mathcal{L}$ and $\mathcal{B}$ be the dgla and BV algebra constructed above, with pairing $(\cdot, \cdot)$.
\begin{prop}
The triple $(\mathcal{L}, \mathcal{B}, (\cdot, \cdot))$ defines an MCP structure on the space of one-dimensional central extensions of $\mathfrak{g}$. 
\end{prop}
The structure is summarised by the following diagram.
\begin{center}
\begin{tikzcd}
  \mathfrak{g}^\ast \arrow[r, "d_\mathfrak{g}"] & \bigwedge \nolimits ^2 \mathfrak{g}^\vee  \arrow[r, "d_\mathfrak{g}"] &  \bigwedge \nolimits ^3 \mathfrak{g}^\vee  \arrow[r, "d_\mathfrak{g}"] & \ldots \\
  \mathfrak{g} \arrow[u ,"\rho_X"]  & \arrow["\delta_\mathfrak{g}" above, l ] \bigwedge \nolimits ^2 \mathfrak{g}  &  \arrow[l, "\delta_\mathfrak{g}" above] \bigwedge \nolimits ^3 \mathfrak{g} &  \arrow[l, "\delta_\mathfrak{g}" above] \ldots
 \end{tikzcd}
 \end{center}
 Where as in Definition~\ref{def:mcp} the map $\rho_X$ is defined as
$$(X, a \wedge b) = (\rho_X a, b),$$
for $a$, $b \in \mathfrak{g}$, $X \in \Lambda^2 \mathfrak{g}^\vee$. Now because ${\rm MC}(\mathcal{L}) = Z^2(\mathfrak{g})$ is a vector space, the MCP structure makes $Z^2(\mathfrak{g})$ a Poisson manifold. Our goal now is to give an explicit characterization of this Poisson structure. 
\begin{lem}
 The cotangent bundle $T^\ast Z^2(\mathfrak{g})$ is trivial and identified as
$$T^\ast Z^2(\mathfrak{g}) \cong Z^2(\mathfrak{g}) \times ( \bigwedge \nolimits ^2 \mathfrak{g} \big / \delta_\mathfrak{g} \bigwedge \nolimits ^3 \mathfrak{g}),$$
\end{lem}
\begin{proof}
At each $X \in Z^2(\mathfrak{g})$ the cotangent space $T_X^\ast Z^2(\mathfrak{g})$ is given by the dual to $Z^2(\mathfrak{g})$. The pairing $(\cdot, \cdot): \bigwedge^n \mathfrak{g} \times \bigwedge^n \mathfrak{g}^\vee \to k$ gives a pairing between homology and cohomology. This pairing is non-degenerate and so the result follows.
\end{proof}
Now given functions $f$, $g \in C^\infty(Z^2(\mathfrak{g}))$ we obtain the Poisson bracket
$$\{f, g\} = (X, \llbracket Df , Dg \rrbracket) = (X, \delta_\mathfrak{g} Df \wedge \delta_\mathfrak{g} Dg) ,$$
where $D$ is the de Rham differential on $Z^2(\mathfrak{g})$. This bracket is the Lie-Poisson bracket associated to the Lie algebra on the cotangent spaces $T^\ast_X Z^2(\mathfrak{g}) = \bigwedge \nolimits ^2 \mathfrak{g} \big / \delta_\mathfrak{g} \bigwedge \nolimits ^3 \mathfrak{g}$. Now pick a non-degenerate, symmetric inner product $\langle \cdot, \cdot \rangle$ on $\mathfrak{g}$, which induces an inner product on $\bigwedge^\bullet \mathfrak{g}$.
\begin{lem}
The inner product $\langle \cdot, \cdot \rangle$ induces an isomorphism (as vector spaces)
$$  \bigwedge \nolimits ^2 \mathfrak{g} \big / \delta_\mathfrak{g} \bigwedge \nolimits ^3 \mathfrak{g} \cong \mathfrak{a}^{b_2} \oplus [\mathfrak{g}, \mathfrak{g}],$$
where $b_2 = {\rm dim}\; H_2(\mathfrak{g})$.  
\end{lem}
\begin{proof}
Using Hodge decomposition we can write
$$ \bigwedge \nolimits ^2 \mathfrak{g} \cong \mathcal{H}_2(\mathfrak{g}) \oplus \delta_\mathfrak{g}\bigwedge \nolimits ^3 \mathfrak{g} \oplus \delta_\mathfrak{g}^\ast \mathfrak{g},$$
where $\delta_\mathfrak{g}^\ast$ is the adjoint of $\delta_\mathfrak{g}$ with respect to the inner product and $\mathcal{H}$ denotes harmonic representatives. Using the Hodge decomposition of $\mathfrak{g}$,
$$ \mathfrak{g} \cong \mathcal{H}_1(\mathfrak{g}) \oplus [\mathfrak{g}, \mathfrak{g}],$$
we find the following isomorphisms, with $\delta^\ast_\mathfrak{g}$ $\delta_\mathfrak{g} = {\rm  Id}$.
\begin{center}
\begin{tikzcd}
   \left [ \mathfrak{g} , \mathfrak{g} \right ]  \arrow[r, "\delta^\ast_\mathfrak{g}",shift left=0.5ex] &  \arrow[l, "\delta_\mathfrak{g}",  shift left=0.5ex]  \delta^\ast_\mathfrak{g} \mathfrak{g}  \end{tikzcd}
 \end{center}
\end{proof}

\begin{lem}
The Lie algebra on $\bigwedge \nolimits ^2 \mathfrak{g} \big / \delta_\mathfrak{g} \bigwedge \nolimits ^3 \mathfrak{g}$ with bracket
$$\llbracket \alpha, \beta \rrbracket =  \delta_\mathfrak{g} \alpha \wedge \delta_\mathfrak{g} \beta,$$
is a central extension of $[\mathfrak{g}, \mathfrak{g}]$ by $\mathfrak{a}^{b_2}$, that is, there is an exact sequence
\begin{center}
\begin{tikzcd}
\mathfrak{a}^{b_2} \arrow[hookrightarrow, r] & \bigwedge \nolimits ^2 \mathfrak{g} \big / \delta_\mathfrak{g} \bigwedge \nolimits ^3 \mathfrak{g} \arrow[r, twoheadrightarrow, "\pi" ] & \left [ \mathfrak{g}, \mathfrak{g} \right ]
 \end{tikzcd}
 \end{center}
\end{lem} 
\begin{proof}
We once again use the Hodge decomposition and write (as vector spaces)
$$ \bigwedge \nolimits ^2 \mathfrak{g} \big / \delta_\mathfrak{g} \bigwedge \nolimits ^3 \mathfrak{g} \cong \mathcal{H}_2(\mathfrak{g}) \oplus \delta_{\mathfrak{g}}^\ast [\mathfrak{g}, \mathfrak{g}],$$
now since elements of $\mathcal{H}_2(\mathfrak{g})$ are $\delta_\mathfrak{g}$ closed, the Lie algebra must have the form
$$ \mathfrak{a}^{b_2} \oplus \mathfrak{k},$$
where $\mathfrak{k}$ is the Lie algebra on elements of $\delta^\ast_\mathfrak{g} [\mathfrak{g}, \mathfrak{g}]$. Now, given $\alpha$, $\beta \in \delta_\mathfrak{g}^\ast [\mathfrak{g}, \mathfrak{g}]$, we may write $\alpha = \delta_\mathfrak{g}^\ast a$ and $\beta = \delta^\ast_\mathfrak{g} b$, for $a$, $b \in [\mathfrak{g}, \mathfrak{g}]$. Then the Lie bracket is given by

$$\delta_\mathfrak{g} \delta_\mathfrak{g}^\ast a \wedge \delta_\mathfrak{g} \delta_\mathfrak{g}^\ast b  = a \wedge b =T+ \delta_\mathfrak{g}^\ast [a,b]$$ 
where $T$ is some linear combination of a harmonic form and a $\delta_\mathfrak{g}$ exact form. Let $\pi$ be the projection onto $\delta_\mathfrak{g}^\ast \mathfrak{g}$, then we see that 
$$\pi (\delta_\mathfrak{g} \delta_\mathfrak{g}^\ast a \wedge \delta_\mathfrak{g} \delta_\mathfrak{g}^\ast b) =  \delta_\mathfrak{g}^\ast [a,b].$$
It follows that there is a short exact sequence for the Lie algebra $ \bigwedge \nolimits ^2 \mathfrak{g} \big / \delta_\mathfrak{g} \bigwedge \nolimits ^3 \mathfrak{g} $ of the form
\begin{center}
\begin{tikzcd}
\mathfrak{a}^{b_2} \arrow[hookrightarrow, r] & \bigwedge \nolimits ^2 \mathfrak{g} \big / \delta_\mathfrak{g} \bigwedge \nolimits ^3 \mathfrak{g} \arrow[r, twoheadrightarrow, "\pi" ] & \left [ \mathfrak{g}, \mathfrak{g} \right ].
 \end{tikzcd}
 \end{center}
\end{proof}
We then obtain the Proposition quoted in the introduction.
\begin{manualprop}{1.1}
Let $\mathfrak{g}$ be a finite-dimensional Lie algebra over either $\mathbb{C}$ or $\mathbb{R}$. The space of central extensions of $\mathfrak{g}$ has a Lie-Poisson structure given by the Lie algebra given defined by the short exact sequence
\begin{center}
\begin{tikzcd}
\mathfrak{a}^{b_2} \arrow[hookrightarrow, r] & \bigwedge \nolimits ^2 \mathfrak{g} \big / \delta_\mathfrak{g} \bigwedge \nolimits ^3 \mathfrak{g} \arrow[r, twoheadrightarrow, "\pi" ] & \left [ \mathfrak{g}, \mathfrak{g} \right ],
 \end{tikzcd}
\end{center}
with the Poisson structure on each gauge orbit being the Lie-Poisson structure associated to $[\mathfrak{g}, \mathfrak{g}]$.
\end{manualprop}

\begin{rem}
Assume the algebra $ \bigwedge \nolimits ^2 \mathfrak{g} \big / \delta_\mathfrak{g} \bigwedge \nolimits ^3 \mathfrak{g}$ is split, so isomorphic to $\mathfrak{a}^{b_2} \oplus [\mathfrak{g}, \mathfrak{g}]$. It then admits the following interpretation. Isomorphism classes of (one-dimensional) central extensions are given by $H^2(\mathfrak{g})$. The space of all extensions can then be thought of as a trivial $[\mathfrak{g}, \mathfrak{g}]$ bundle over $H^2(\mathfrak{g})$, corresponding to gauge orbits. The MCP structure corresponds to the  Poisson structure arising from taking the $[\mathfrak{g}, \mathfrak{g}]$ Lie-Poisson structure on each orbit.
\end{rem}

\section{Poisson structures} \label{sec:poisson}

This case was already considered~\cite{machon2020poisson}. Let $M$ be a smooth, closed, orientable manifold and $\mu$ a volume form on $M$. A Poisson structure on $M$ is a 2-vector $\Pi \in \Gamma(\bigwedge^2 TM)$ satisfying
$$[\Pi, \Pi]_S=0,$$
where $[\cdot, \cdot]_S$ is the Schouten-Nijenhuis bracket. This can be studied in the dgla formalism (see e.g.~\cite{kontsevich2003deformation}) as follows. For the vector space $L$ we take sections of the exterior algebra of the tangent bundle, shifted by -1,
$$ L = \Gamma \left (\bigwedge^\bullet TM  \right )[-1]= \bigoplus_i \Gamma \left (  \bigwedge^i TM \right)[-1]$$
The bracket is the Schouten-Nijenhuis bracket $[\cdot, \cdot]_S$, which gives $L$ the structure of a graded Lie algebra. The base differential is trivial, and so the triple
$$ \mathcal{L} = (L, 0, [\cdot, \cdot]_S),$$
defines our base dgla. The set
$$\ldef = \left \{ X \in  \Gamma(\bigwedge \nolimits^2 TM) \, | \,  [X, X]_S=0 \right \}$$
then corresponds to Poisson structures on $M$. Since $d=0$ in $\mathcal{L}$, $\ldef$ is a cone, so property 1 of the MCP structure is satisfied. For the commutative algebra $\mathcal{B}$ we take the differential forms on $M$,
$$ B = \bigoplus_i \Omega^i(M),$$
equipped with the wedge product $\wedge$. For the pairing between $L$ and $B$ we pick a volume form $\mu$ on $M$ and take the pairing between a $p$-vector $X$ and $p$-form $\alpha$ given by
\begin{equation} \label{eq:poipair}(X, \alpha) = (-1)^p \int_M (\iota_X \alpha) \mu .
\end{equation}
\begin{rem}
The factor of $(-1)^p$ ensures the sign conventions match with the rest of the paper, but is not essential to the construction.
\end{rem}
\begin{lem}[Ref.~\cite{machon2020poisson}, Lemma 2]
Given a Poisson structure $X$, the adjoint of the differential $d_X=[X, \cdot]_S$ is given by
$$\delta_X = [ \iota_X, d_{DR} ] - \iota_{\phi},$$
where $d_{DR}$ is the de Rham differential on $M$ and $\phi$ is the modular vector field of the Poisson structure~\cite{weinstein1997modular}, defined as
$$ \iota_\phi \mu = d \iota_X \mu.$$
\end{lem}
\begin{lem}
For each $X \in \ldef$, the triple $(B, \delta_X, \wedge)$ forms a BV algebra.
\end{lem}
\begin{proof}
The fact that this defines a BV algebra for $\phi=0$ is well-known~\cite{kosmann1995exact}. $\phi$ is a Poisson vector field, $[\phi, X]_S=0$, and $\iota_\phi$ is a derivation of $\Omega^\bullet(M)$, which together imply that the order-2 identity is not affected.
\end{proof}
The last thing to check is property 3. Since a Poisson structure on $M$ defines a Poisson Lie algebroid on $M$, it follows that the map $\rho_X$ (which is the induced map $\Pi^\# : \Gamma(T^\ast M) \to \Gamma(TM)$) is a homomorphism from the Gerstenhaber bracket on the BV algebra to the Lie bracket of vector fields, hence we have $k_X=1$ for property 3 of the MCP structure.
\begin{prop}
The triple $(\mathcal{L}, \mathcal{B}, (\cdot, \cdot)$ defines an MCP structure.
\end{prop}
\begin{proof}
The only additional requirement (see Section~\ref{sec:vb}) is to note that since gauge transformations are diffeomorphisms and $M$ is closed, for any $V \in L^0$ the gauge transformation by $V$ exists.
\end{proof}

\begin{rem}
Properties of this MCP structure are explored in detail in Ref.~\cite{machon2020poisson}.
\end{rem}

\subsection{The cotangent and isotropy Lie algebras}

Given the MCP structure, we can explore the Lie algebras defined for gauge orbits in Sections~\ref{sec:alg} and~\ref{sec:cot}. Given a Poisson structure $X$, its gauge orbits are those Poisson structures $X^\prime$ obtained from $X$ via diffeomorphism. We can then identity the cotangent space at $X$ of the gauge orbit containing $X$ as the quotient $(\Omega^2(M)/Z(\Omega^2(M)))$ (recall $Z(\Omega^\bullet(M)$ is the $\delta_X$-closed differential forms). We then obtain the following as a consequence of Proposition~\ref{prop:gla}. \begin{cor}
For each Poisson structure $X \in \ldef$ there is a graded Lie algebra on $(\Omega^\bullet(M)/Z(\Omega^\bullet(M)))[-2]$ given by
$$\llbracket a, b \rrbracket = (-1)^{|a|} \delta_X a \wedge \delta_X b,$$
the subalgebra $\Omega^2(M)/Z(\Omega^3(M))$ is the cotangent algebra $\mathfrak{c}_X$.
\end{cor}
We can also study the algebra $\mathfrak{h}_X$. In the case of Poisson structures there is a graded version. For each $p$, the Poisson structure $X$ induces a map $X^\#: \Omega^p \to \Gamma(\bigwedge^p TM)$, given by extending the map on 1-forms $a \mapsto X(a, \cdot)$ (where we think of $X$ as a tensor) in the natural way. Then define the graded vector space 
$$ K = \left \{ \alpha \in \Omega^\bullet(M) / Z( \Omega^\bullet(M)) \; | \; [ X^\# \delta_X \alpha, X]=0  \right \}$$
as differential forms mapping to multivectors that commute with the Poisson structure $X$. For 2-forms this reduces to the map $\nu_X$ discussed in Section~\ref{sec:cot}.
\begin{lem}
$K[-2]$ is a graded Lie subalgebra of $(\Omega^\bullet(M)/Z( \Omega^\bullet(M)))[-2]$, 
and the subalgebra corresponding to 2-forms is the Lie algebra $\mathfrak{h}_X$ of the Poisson structure $X$.
\end{lem}
\begin{proof}
Observe that $X^\# (\alpha \wedge \beta) = (X^\# \alpha) \wedge(X^\# \beta)$, where the second wedge product is the product of multivector fields. The wedge product distributes over the Schouten-Nijenhuis bracket
$$[X, A \wedge B ]_S = [X, A]_S\wedge B + (-1)^{(|X|-1)|A|} A \wedge [X, B]_S,$$
(where we take their gradings in as multivector fields) it follows then that if $\alpha$, $\beta$ are in $K$, then $\llbracket \alpha, \beta \rrbracket \in K$ also.
\end{proof}
In Section~\ref{sec:alg} we also defined the isotropy algebra $\mathfrak{j}_X$ at $X$ of the Lie-Poisson algebroid on the gauge orbit containing $X$. Although we do not define the algebroid in this case, we can still compute $\mathfrak{j}_X$.
\begin{lem} \label{lem:poiiso}
For $\alpha$, $\beta \in K^2$, the isotropy algebra of $X$ is given by the bracket
$$ [a, b]_{\mathfrak{j}_X} = \llbracket \alpha, \beta \rrbracket+ \mathcal{L}_{V_a} b - \mathcal{L}_{V_b} a,$$
where $V_a= \rho_X \delta_X a$.
\end{lem}

\begin{proof}
From Proposition~\ref{prop:algebras} the isotropy algebra is given by the bracket
$$ \llbracket a, b \rrbracket +O_{X,a} b- O_{X,b} a,$$
for $a$, $b \in \Omega^2(M) / Z(\Omega^3(M))$. We recall the definition of the $O$ operator as
$$(A, O_{X,a} b) = ([A, \rho_X \delta_X a], b) = ([A, V_a], b),$$
for an arbitrary $2$-vector $A$, where we define the the vector field $V_a$. This is given explicitly as
$$ \int_M (\iota_{[A,V_a]} b) \mu.$$
Using the Cartan identity we have
$$\iota_{[A,V_a]}  = \iota_A \mathcal{L}_{V_a} - \iota_{V_a} d \iota_{A},$$
which gives
$$ \int_M (\iota_{A} \mathcal{L}_{V_a} b) - \iota_{V a} d (\iota_{A} b) \mu,$$
now $V_a$ is volume-preserving, and an application of Stokes' theorem shows that the second term vanishes.
\end{proof}

\section{Symplectic structures} \label{sec:examsym}

It is worth considering symplectic structures separately from Poisson structures. We construct two separate MCP structures in this section. The first is, essentially, an MCP structure on the space of 2-forms, which becomes an MCP structure on symplectic forms when the 2-form is non-degenerate, and in this case reproduces the MCP algebroid of Section~\ref{sec:poisson}  for Poisson structures given above. 

The second MCP structure we construct returns to the Poisson perspective (i.e, the dgla consists of multivector fields not differential forms), but assumes the Poisson structure is symplectic. In that case there is a natural choice of volume form, the symplectic volume, and we can construct an MCP structure (an example of a compatible sub-dgla) for symplectic forms compatible with the volume form. The isotropy algebras are particularly interesting in this case, and we obtain a number of invariants for symplectic manifolds. One could try and treat this constant volume case from the symplectic perspective (i.e. using differential forms as the dgla), but then the volume constraint becomes non-linear, and difficult to deal with. On the other hand, this suggests that more elaborate formulations of MCP structures exist, but we do not explore these ideas further.

\begin{rem}
As pointed out by Liu, Weinstein and Xu~\cite{liu1997manin}, the MC conditions for Poisson and symplectic structures each have only one of the two conditions in the MC equation, with a unified perspective arising through the use of Dirac structures and Courant algebroids. While it may be possible to construct an MCP structure for deformations of Dirac structures, we do not consider it here.
\end{rem}

\subsection{Deformations of symplectic structures}

 Let $M$ be a closed manifold of dimension $2n$, and fix a volume form $\mu$. We now construct an MCP structure for 2-forms on $M$. For the base dgla $\mathcal{L}$ we take the shifted de Rham complex $\mathcal{L}=(\Omega^\bullet(M)[-1],d)$ with trivial Lie bracket. The Maurer-Cartan set then consists of closed 2-forms
$$ {\rm MC}(\mathcal{L})  =\left  \{ \omega \in \Omega^2(M) \; | \; d\omega = 0 \right \} ,$$
note that the tangent dglas do not depend on $\omega$, so that tangent dglas are all equal, $\mathcal{L}_\omega = \mathcal{L}$.

For the BV algebra we take the algebra of multivector fields
$$ A^\bullet = \bigoplus A^i, \quad A^i = \Gamma \left (\bigwedge \nolimits ^i TM \right ),$$
with product given by the wedge product, giving $A^\bullet$ the structure of a graded commutative algebra. As in Section~\ref{sec:poisson}, the pairing is given by the pairing between a p-form $\alpha \in \Omega^p(M)$ and p-vector $V \in A^p(M)$ as
$$(\alpha, V)= (-1)^{p} \int_M \iota_V \alpha \mu.$$
\begin{rem}
Note the pairing is reversed in order when compared with \eqref{eq:poipair}. Our convention is to place the dgla on the left, which here is differential forms. We keep the sign convention.
\end{rem}
The differential $\delta:A^i \to A^{i-1}$ is then defined as the adjoint of the de Rham differential with respect to this pairing. The volume form induces an isomorphism $\ast_\mu : A^p \to \Omega^{2n-p}$ given by $\ast_\mu Y = \iota_Y \mu$.
\begin{lem}
The differential $\delta: A^p \to A^{p-1}$ is given by
$$ \delta =- \ast_\mu^{-1} d \ast_\mu$$
\end{lem}
\begin{proof}
Take $V \in A^p(M)$ and $\alpha \in \Omega^{p-1}(M)$, then we have
$$(d\alpha, V) = \int_M (\iota_V d \alpha) \mu = (-1)^{p+1} \int_M d \alpha \wedge \iota_V \mu = (-1)^{p+1} \int_M d \alpha \wedge \ast_\mu (V)$$
Using Stokes' theorem we have
$$(-1)^{p+1} \int_M d \alpha \wedge \ast_\mu (V) = - \int_M \alpha \wedge  d  \ast_\mu (V)  = - \int_M \alpha \wedge( \ast_\mu   \ast_\mu^{-1} d\ast_\mu(V) )\mu = \int_M (\iota_{\delta V} \alpha) \mu.$$

\end{proof}

That this differential defines a BV algebra $\mathcal{B}$ on multivector fields is well-known, see e.g.~\cite{roger2009gerstenhaber}, and we will not reproduce the proof of the order-2 property. Finally, note that the homomorphism $ A^1 \to \Omega^1$ is just given by the isomorphism $\omega^\#: Y \mapsto \omega(Y, \cdot)$. Since the Lie bracket on the base dgla is trivial, we have.
\begin{prop}
The triple consisting of $(\mathcal{L}, \mathcal{B}, (\cdot, \cdot))$ defines an MCP structure whose MC elements are closed 2-forms.
\end{prop}
\begin{proof}
Following arguments in Section~\ref{sec:vb} we need only establish that gauge transformations exist. Here they are diffeomorphisms of a compact manifold along constant vector fields (similarly to the Poisson case).
\end{proof}
\begin{rem}
This is an infinite-dimensional Lie-Poisson structure, and is strongly reminiscent of those found in geometric approaches to hydrodynamics~\cite{arnold2021topological, marsden2013introduction,morrison1998hamiltonian}.
\end{rem}

Now let $[F]$ be an admissible function (as defined in Section~\ref{sec:vb}), with representative $F$, then its derivative at $\alpha \in \Omega^2(M)$ will be a 2-vector $W_F$. By Lemma~\ref{lem:fdiff} the vector field $V_F = \delta_X W_F$ is well-defined (i.e. does not depend on the representative $F$).
Then Proposition~\ref{prop:flow} gives us.
\begin{cor} \label{cor:symflow}
The flow equation of the MCP structure on closed 2-forms acts by volume preserving diffeomorphisms, with
$$ \frac{d \alpha}{dt} = d V_F = d \delta_X W_F = \mathcal{L}_{V_F} \alpha.$$
\end{cor}
\begin{cor}
Non degeneracy $\alpha^n \neq 0$ is preserved by the flow, hence symplectic structures on $M$ form a collection of gauge orbits of the MCP structure on closed 2-forms.
\end{cor}
Now suppose we are on a gauge orbit containing symplectic structures, then we can map $W_F$ to a 2-form $\beta_F$ given by applying invertible map from vectors to forms defined by the symplectic structure. Upon doing so one finds that the flow equation of a symplectic structure $\omega$ is written as
$$ \frac{d \omega}{dt} = dd^\Lambda \beta_F$$
where the symplectic differential $d^\Lambda = [\iota_\Pi, d ]$, with $\Pi$ the Poisson structure corresponding to $\omega$.
\begin{rem}
The flow by a $dd^\Lambda$ exact element is reminiscent of the geometric flow for the type IIA string defined by Fei et al.~\cite{fei2020geometric}, building on the work of Hitchin~\cite{hitchin2000geometry}.
\end{rem}
Further rewriting the flow equation in terms of the Poisson structure, one sees that it recovers the flow equation in Section~\ref{sec:poisson}. We do not explore the isotropy algebras in this case, delaying their study to the subsequent section.

\subsection{Constant volume deformations of symplectic structures} \label{sec:cv}

The pairing for the MCP structures construct for Poisson and symplectic structures requires a choice of volume form. In the symplectic case there is a natural choice, we can take the symplectic volume,
$$ \mu = \frac{\omega^n}{n!}$$
and since the flow of the bracket is by volume-preserving diffeomorphisms, we should be able to define an MCP structure for symplectic structures whose symplectic volume form is $\mu$. In fact, this will be an example of a compatible sub-dgla, (recall Definition~\ref{def:comp}).

It is easier to study this case from the Poisson perspective, as we will now demonstrate. A symplectic form $\omega$ having symplectic volume $c \mu$, with $c \in \mathbb{R}$ a constant is equivalent to the statement that
$d \iota_\pi \mu=0,$
where $\pi$ is the Poisson structure defined as the inverse of $\omega$. If we fix $\mu$, this is a linear constraint on the Poisson structure, which is easily dealt with. This MCP structure was also considered in detail in~\cite{machon2020poisson}. We note also that deformations of symplectic structures from this perspective were considered by De Bartolomeis~\cite{bartolomeis2005z}.

The idea is to take the sub-dgla of multivector fields corresponding to the volume-preserving subalgebra. So, for the base dgla $\mathcal{L}$ we take the vector space of multivector fields $X$ satisfying $d\iota_
X \mu = 0$, with Schouten-Nijenhuis bracket $[\cdot, \cdot]_S$ and trivial differential. That is, let 
$$A = \left \{ X \in \Gamma \left (T\bigwedge^\bullet M \right) \;|\; d \iota_X \mu=0 \right\}.$$
That this space is closed under the Schouten-Nijenhuis bracket follows from the Cartan identity
$$\iota_{[P,Q]_S} = [[\iota_P, d], \iota_Q],$$
where here the bracket $[\cdot, \cdot]$ corresponds to the graded commutator of derivations on $\Omega^\bullet(M)$.
\begin{lem}
The triple $\mathcal{A} = (A[-1],0, [\cdot, \cdot]_S)$ is a sub-dgla, of the dgla of multivector fields, and is compatible with the MCP structure considered in Section~\ref{sec:poisson}.
\end{lem}
\begin{proof}
That the triple is a dgla follows from the above considerations. That it is compatible with the MCP structure follows from the fact that the hamiltonian flow of the MCP structure given in Section~\ref{sec:poisson} acts by volume-preserving diffeomorphisms.
\end{proof}
\begin{cor}
The dgla $\mathcal{L}$ is a compatible sub-dgla of the MCP structure associated to deformations of Poisson structures constructed in Section~\ref{sec:poisson}.
\end{cor}

Now the above construction holds for all Poisson structure compatible with the volume-form $\mu$. We restrict ourselves now to a single gauge orbit of some symplectic structure $\omega$.

\subsection{Isotropy algebras for symplectic structures}

We now explore the cotangent Lie algebras for the volume-preserving deformations of symplectic structures. Because this is defined through a compatible sub-dgla we only have one Lie algebra, the isotropy algebra $\mathfrak{j}_X$ of the algebroid associated to gauge orbits, which in this case are generated by volume-preserving diffeomorphisms. Now to define the isotropy Lie algebra we must consider a single gauge orbit $\mathcal{O}$.
\begin{lem}
At $X \in \mathcal{O}$ the cotangent space $T^\ast_X \mathcal{O}$ is given as
$$ T^\ast_X \mathcal{O} \cong \Omega^2(M) / (d \Omega^1(M) + Z^2(\Omega^\bullet(M))),$$
where $Z^2(\Omega^\bullet(M))$ refers to $\delta_X$-closed 2-forms on $M$.
\end{lem}
\begin{proof}
The cotangent space of the orbit is defined as the dual of the tangent space with respect to the pairing. The tangent space of the orbit consists of all 2-vector fields $T = [X, V]_S$, where $V$ is a volume-preserving vector field. The tangent space will therefore be of the form $\Omega^2(M) / N$, where $N$ is the subspace satisfying
$$(T,  n)=0,$$
for all $A \in T_X \mathcal{O}$ and $n \in N$, now assuming $T$ is of the form $[X,V]_S$ we have
$$ ([X,V]_S, n) = (V, \delta_X n)$$
which implies $n$ contains all $\delta_X$ closed elements. Accounting for the volume-preserving aspects requires 
$$(T,  n)=0,$$
for all $T \in A^2$. A short calculation shows that this implies $n$ is $d$-exact (this uses the non-degeneracy implied by Poincar\'{e} duality). Hence $N$ is the space $d \Omega^1(M) + Z^2(\Omega^\bullet(M))$.
\end{proof}
Now the kernel of the map $\nu_X$ is given by $d \delta_X$-closed forms. Hence we expect a Lie algebra on $ \left ( {\rm ker} \; d \delta_X \cap \Omega^2(M) \right ) / (d \Omega^1(M) + Z^2(\Omega(M)))$. 
\begin{lem}
The bracket for the isotropy Lie algebra $\mathfrak{j}_X$ on 
$$ \left ( {\rm ker} \; d \delta_X \cap \Omega^2(M) \right ) / (d \Omega^1(M) + Z^2(\Omega(M)))$$
is given by the formula
$$[a, b]_{\mathfrak{j}_X} = - \llbracket a, b \rrbracket.$$
\end{lem}
\begin{rem}
Note the minus sign.
\end{rem}
\begin{proof}
From Lemma~\ref{lem:poiiso} we have
$$[a, b]_{\mathfrak{j}_X}  = \llbracket a, b \rrbracket + \mathcal{L}_{V_a} b - \mathcal{L}_{V_b} a.$$
Now we can write the second and third terms as
\begin{align*} &\mathcal{L}_{V_a} b - \mathcal{L}_{V_b}a = -2 d^\Lambda a \wedge d^\Lambda b \\
+ &\iota_X(d^\Lambda a \wedge d b - d^\Lambda b \wedge da) + d \iota_X(d^\Lambda a \wedge b - d^\Lambda b \wedge a) + d d^\Lambda b \wedge \iota_X a - d d^\Lambda a \wedge \iota_X b
\end{align*}
Now using the fact that $d d^\Lambda a = 0$, as $a$ is in the kernel of $\nu_X$ (similarly for $b$) we find the expression
$$ \mathcal{L}_{V_a} b - \mathcal{L}_{V_b}a = -2 d^\Lambda a \wedge d^\Lambda b - d^\Lambda ( d^\Lambda a \wedge b) + d^\Lambda(d^\Lambda b \wedge a),$$
but the last two terms are elements of $Z^2(\Omega(M))$, hence equivalent to zero.

\end{proof}
Now we would like to explore the properties of this algebra, we will neglect the minus sign for notational convenience. Recall that the symplectic Aeppli or $d d^\Lambda$ cohomology is given by (see~\cite{tseng2012cohomology, angella2015inequalities})
$$H_{dd^\Lambda}^\bullet(M) = \frac{{\rm ker}\;  d d^\Lambda }{{\rm im}\; d+ d^\Lambda }.$$
and we also define the cohomology groups
$$H_{d^\Lambda}^\bullet(M) = \frac{{\rm ker}\;  d^\Lambda }{{\rm im}\;  d^\Lambda }.$$
It is worth recalling some of the properties of these groups (see e.g.~\cite{tardini2019symplectic}, Theorem 2.1, also Refs.~\cite{tseng2012cohomology, angella2015inequalities}).
\begin{prop}
The groups $H_{dd^\Lambda}^\bullet(M) $ and $H_{d^\Lambda}^\bullet(M)$ satisfy
$$ {\rm dim}\; H_{dd^\Lambda}^i(M)  = h_{dd^\Lambda}^i \geq b_i(M) = {\rm dim}\; H_{d^\Lambda}^i(M), $$
where $b_i(M)$ is the $i^{\rm th}$ Betti number of $M$. The natural map
$$ H^i_{d^\Lambda}(M)  \to H_{dd^\Lambda}^i(M),$$
is an isomorphism for all $i$ if and only if $(M , \omega)$ satisfies the hard Lefschetz condition.
\end{prop}

We can then characterise the algebra $\mathfrak{j}_X$ as follows.
\begin{lem} \label{lem:sympvec}
There is a short exact sequence of Lie algebras
\begin{center}
\begin{tikzcd}
{ H^2_{d^\Lambda}(M)} \arrow[hookrightarrow, r] & H^2_{d d^\Lambda}(M) \arrow[r, twoheadrightarrow ] & \mathfrak{j}_X,
 \end{tikzcd}
 \end{center}
 with bracket $\llbracket \cdot, \cdot \rrbracket$.
\end{lem}
\begin{proof}
There is a short exact sequence of vector spaces
 \begin{center}
\begin{tikzcd}
\frac{{\rm ker} \; d^\Lambda \cap \Omega^2(M)}{ d^\Lambda \Omega^3(M)} \arrow[hookrightarrow, r] &\frac{{\rm ker} \; d d^\Lambda \cap \Omega^2(M)}{ d^\Lambda \Omega^3(M)+ d \Omega^1(M)}  \arrow[r, twoheadrightarrow ] &\frac{{\rm ker} \; d d^\Lambda \cap \Omega^2(M)}{ {\rm ker} \; d^\Lambda \cap \Omega^2(M)+ d \Omega^1(M)},
 \end{tikzcd}
 \end{center}
and note that $H^2_{d^\Lambda}$ is abelian subalgebra in the center of $H^2_{d d^\Lambda}(M)$. 
\end{proof}
Before exploring the properties of $\mathfrak{j}_X$, we first show that it is part of a larger structure.
\begin{thm} \label{thm:sym}
Let $(M, \omega)$ be a closed symplectic manifold. Then there is an exact sequence of finite-dimensional graded Lie algebras,
\begin{center}
\begin{tikzcd}
{ H^\bullet_{d^\Lambda}(M)}[-2] \arrow[hookrightarrow, r] & H^\bullet_{d d^\Lambda}(M)[-2]\arrow[r, twoheadrightarrow ] & \mathfrak{k},
 \end{tikzcd}
 \end{center}
 with bracket
 $$ \llbracket a, b \rrbracket = (-1)^{|a|} d^\Lambda a \wedge d^\Lambda b.$$ $H^\bullet_{d^\Lambda}(M)[-2] $ is abelian and the algebras $H^\bullet_{d d^\Lambda}(M)[-2]$ and $\mathfrak{k}$ are two-step nilpotent, and $\mathfrak{k}$ is zero-dimensional if and only if $(M, \omega)$ satisfies the hard Lefschetz condition. $\mathfrak{k}^0  =\mathfrak{j}_\omega$, the isotropy algebra at $\omega$.
\end{thm}
\begin{proof}
The structure of the vector spaces follows from analogous arguments to Lemma~\ref{lem:sympvec}. The graded Jacobi identity follows from calculations in Section~\ref{sec:cot}. To see the two-step nilpotence observe that
\begin{align*} & \llbracket a, \llbracket b, c \rrbracket \rrbracket = (-1)^{|a|+|b|} d^\Lambda a \wedge d^\Lambda (d^\Lambda b \wedge d^\Lambda c) \\
= & (-1)^{|a|+|b|+1}  d^\Lambda a \wedge d \Lambda (d^\Lambda b \wedge d^\Lambda c) \\
= & (-1)^{|b|+1}d ( d^\Lambda a \wedge  \Lambda (d^\Lambda b \wedge d^\Lambda c)  ) \sim 0.
\end{align*}
It is immediately seen that the subalgebra algebra ${ H^\bullet_{d^\Lambda}(M)}[-2]$ is an abelian ideal of  ${ H^\bullet_{dd^\Lambda}(M)}[-2]$, hence the quotient inherits the structure of a Lie algebra. Finite-dimensionality follows from the results of Tseng and Yau~\cite{tseng2012cohomology} (see e.g. Proposition 4.4). 
That $\mathfrak{k}$ is zero-dimensional if and only if the hard Lefschetz condition is satisfied follows from results in e.g. Ref.~\cite{tardini2019symplectic}, Theorem 2.1 also Ref.~\cite{merkulov1998formality}

\end{proof}

\begin{lem}
The algebra $\mathfrak{k}$ is supported in degrees $0,1, \ldots, 2n-4$,
$$\mathfrak{k} = \bigoplus_{i=0}^{2n-4} \mathfrak{k}^i.$$
\end{lem}
\begin{proof}
The maps from $H^i_{d^\Lambda}$ to $H^i_{d d^\Lambda}$ are always isomorphisms in degrees $i$ equal to 0, 1, $2n$ and $2n-1$, accounting for the shifting gives the result.
\end{proof}
In particular, if $(M, \omega)$ is a symplectic 4-manifold we obtain a single Lie algebra which vanishes if and only if $(M, \omega)$ satisfies the hard Lefschetz condition. We can now return to the discussion of the gauge orbits, and we obtain the corollary of Theorem~\ref{thm:sym} quoted in the introduction.
\begin{manualtheorem}{1.3}
The Poisson structure on a gauge orbit $\mathcal{O}$ is symplectic if and only if the invariant $h^2_{d d^\Lambda}$ vanishes.
\end{manualtheorem}
\begin{rem}
It would be interesting to see examples for which the Lie algebra $\mathfrak{k}$ is non-trivial.
\end{rem}

\bibliographystyle{plain}
\bibliography{def_vor_liv.bib}

\begin{thebibliography}{10}

\bibitem{angella2015inequalities}
Daniele Angella and Adriano Tomassini.
\newblock Inequalities {\`a} la {F}r{\"o}licher and cohomological
  decompositions.
\newblock {\em Journal of Noncommutative Geometry}, 9(2):505--542, 2015.

\bibitem{arnold2021topological}
Vladimir~I Arnold and Boris~A Khesin.
\newblock {\em Topological methods in hydrodynamics}, volume 125.
\newblock Springer Nature, 2021.

\bibitem{bartolomeis2005z}
Paolo~de Bartolomeis.
\newblock $\mathbb{Z}_2$ and $\mathbb{Z}$-deformation theory for holomorphic
  and symplectic manifolds.
\newblock In {\em Complex, contact and symmetric manifolds}, pages 75--103.
  Springer, 2005.

\bibitem{bergeron1995decomposition}
Nantel Bergeron and H~Lewis Wolfgang.
\newblock The decomposition of {H}ochschild cohomology and {G}erstenhaber
  operations.
\newblock {\em Journal of Pure and Applied Algebra}, 104(3):243--265, 1995.

\bibitem{cartan2016homological}
Henry Cartan and Samuel Eilenberg.
\newblock {\em Homological Algebra (PMS-19), Volume 19}.
\newblock Princeton university press, 2016.

\bibitem{costello2021factorization}
Kevin Costello and Owen Gwilliam.
\newblock {\em Factorization algebras in quantum field theory}, volume~2.
\newblock Cambridge University Press, 2021.

\bibitem{crainic2004integrability}
Marius Crainic and Rui~Loja Fernandes.
\newblock Integrability of {P}oisson brackets.
\newblock {\em Journal of Differential Geometry}, 66(1):71--137, 2004.

\bibitem{fei2020geometric}
Teng Fei, Duong~H Phong, Sebastien Picard, and Xiangwen Zhang.
\newblock Geometric flows for the type {IIA} string.
\newblock {\em arXiv preprint arXiv:2011.03662}, 2020.

\bibitem{fernandes2003invariants}
Rui~Loja Fernandes.
\newblock Invariants of {L}ie algebroids.
\newblock {\em Differential Geometry and its Applications}, 19(2):223--243,
  2003.

\bibitem{gerstenhaber1963cohomology}
Murray Gerstenhaber.
\newblock The cohomology structure of an associative ring.
\newblock {\em Annals of Mathematics}, pages 267--288, 1963.

\bibitem{gerstenhaber1964deformation}
Murray Gerstenhaber.
\newblock On the deformation of rings and algebras.
\newblock {\em Annals of Mathematics}, pages 59--103, 1964.

\bibitem{gerstenhaber1987hodge}
Murray Gerstenhaber and Samuel~D Schack.
\newblock A {H}odge-type decomposition for commutative algebra cohomology.
\newblock {\em Journal of Pure and Applied Algebra}, 48(1-2):229--247, 1987.

\bibitem{gerstenhaber1991shuffle}
Murray Gerstenhaber and Samuel~D Schack.
\newblock The shuffle bialgebra and the cohomology of commutative algebras.
\newblock {\em Journal of Pure and Applied Algebra}, 70(3):263--272, 1991.

\bibitem{getzler2002darboux}
Ezra Getzler.
\newblock A {D}arboux theorem for {H}amiltonian operators in the formal
  calculus of variations.
\newblock {\em Duke Mathematical Journal}, 111(3):535--560, 2002.

\bibitem{goldman1988deformation}
William~M Goldman and John~J Millson.
\newblock The deformation theory of representations of fundamental groups of
  compact {K}{\"a}hler manifolds.
\newblock {\em Publications Math{\'e}matiques de l'IH{\'E}S}, 67:43--96, 1988.

\bibitem{hamilton1982inverse}
Richard~S Hamilton.
\newblock The inverse function theorem of {N}ash and {M}oser.
\newblock {\em Bulletin (New Series) of the American Mathematical Society},
  7(1):65--222, 1982.

\bibitem{hinich1996descent}
Vladimir Hinich.
\newblock Descent of {D}eligne groupoids.
\newblock {\em arXiv preprint alg-geom/9606010}, 1996.

\bibitem{hitchin2000geometry}
Nigel Hitchin.
\newblock The geometry of three-forms in six dimensions.
\newblock {\em Journal of Differential Geometry}, 55(3):547--576, 2000.

\bibitem{kirillov2004lectures}
Aleksandr~Aleksandrovich Kirillov.
\newblock {\em Lectures on the orbit method}, volume~64.
\newblock American Mathematical Soc., 2004.

\bibitem{kontsevich2003deformation}
Maxim Kontsevich.
\newblock Deformation quantization of {P}oisson manifolds.
\newblock {\em Letters in Mathematical Physics}, 66(3):157--216, 2003.

\bibitem{kosmann1995exact}
Yvette Kosmann-Schwarzbach.
\newblock Exact {G}erstenhaber algebras and {L}ie bialgebroids.
\newblock In {\em Geometric and algebraic structures in differential
  equations}, pages 153--165. Springer, 1995.

\bibitem{kosmann1990poisson}
Yvette Kosmann-Schwarzbach and Franco Magri.
\newblock {P}oisson-{N}ijenhuis structures.
\newblock In {\em Annales de l'IHP Physique th{\'e}orique}, volume~53, pages
  35--81, 1990.

\bibitem{koszul1985crochet}
Jean-Louis Koszul.
\newblock Crochet de {S}chouten-{N}ijenhuis et cohomologie.
\newblock {\em Ast{\'e}risque}, 137(257-271):4--3, 1985.

\bibitem{lian1993new}
Bong~H Lian and Gregg~J Zuckerman.
\newblock New perspectives on the {BRST}-algebraic structure of string theory.
\newblock {\em Communications in Mathematical Physics}, 154(3):613--646, 1993.

\bibitem{liu1997manin}
Zhang-Ju Liu, Alan Weinstein, and Ping Xu.
\newblock {M}anin triples for {L}ie bialgebroids.
\newblock {\em Journal of Differential Geometry}, 45(3):547--574, 1997.

\bibitem{loday2013cyclic}
Jean-Louis Loday.
\newblock {\em Cyclic homology}, volume 301.
\newblock Springer Science \& Business Media, 2013.

\bibitem{machon2020poisson}
Thomas Machon.
\newblock A {P}oisson bracket on the space of {P}oisson structures.
\newblock {\em Journal of Symplectic Geometry, accepted}, arXiv preprint
  arXiv:2008.11074.

\bibitem{manetti2004deformation}
Marco Manetti.
\newblock Lectures on deformations of complex manifolds.
\newblock {\em Rendiconti di Matematica}, 24:1--183, 2004.

\bibitem{marsden2013introduction}
Jerrold~E Marsden and Tudor~S Ratiu.
\newblock {\em Introduction to mechanics and symmetry: a basic exposition of
  classical mechanical systems}, volume~17.
\newblock Springer Science \& Business Media, 2013.

\bibitem{merkulov1998formality}
Sergei~A Merkulov.
\newblock Formality of canonical symplectic complexes and {F}robenius
  manifolds.
\newblock {\em arXiv preprint math/9805072}, 1998.

\bibitem{morrison1998hamiltonian}
Philip~J Morrison.
\newblock Hamiltonian description of the ideal fluid.
\newblock {\em Reviews of modern physics}, 70(2):467, 1998.

\bibitem{nijenhuis1967deformations}
Albert Nijenhuis and RW~Richardson.
\newblock Deformations of {L}ie algebra structures.
\newblock {\em Journal of Mathematics and Mechanics}, 17(1):89--105, 1967.

\bibitem{pridham2010unifying}
Jon~P Pridham.
\newblock Unifying derived deformation theories.
\newblock {\em Advances in Mathematics}, 224(3):772--826, 2010.

\bibitem{quillen1970co}
Daniel Quillen.
\newblock On the (co-) homology of commutative rings.
\newblock In {\em Proc. Symp. Pure Math}, volume~17, pages 65--87, 1970.

\bibitem{roger2009gerstenhaber}
Claude Roger.
\newblock {G}erstenhaber and {B}atalin-{V}ilkovisky algebras; algebraic,
  geometric, and physical aspects.
\newblock {\em Archivum Mathematicum}, 45(4):301--324, 2009.

\bibitem{tardini2019symplectic}
Nicoletta Tardini and Adriano Tomassini.
\newblock Symplectic cohomologies and deformations.
\newblock {\em Bollettino dell'Unione Matematica Italiana}, 12(1):221--237,
  2019.

\bibitem{tseng2012cohomology}
Li-Sheng Tseng and Shing-Tung Yau.
\newblock Cohomology and {H}odge theory on symplectic manifolds: I.
\newblock {\em Journal of Differential Geometry}, 91(3):383--416, 2012.

\bibitem{weinstein1997modular}
Alan Weinstein.
\newblock The modular automorphism group of a {P}oisson manifold.
\newblock {\em Journal of Geometry and Physics}, 23(3-4):379--394, 1997.

\bibitem{zhu2014co}
Can Zhu, Fred Van~Oystaeyen, and Yinhuo Zhang.
\newblock On (co) homology of {F}robenius {P}oisson algebras.
\newblock {\em Journal of K-Theory}, 14(2):371--386, 2014.

\end{thebibliography}

\end{document}